\documentclass[reqno]{amsart}
\usepackage{amssymb}
\usepackage{amsmath}
\usepackage[mathscr]{euscript}
\usepackage{times}
\usepackage{newcent}       % selects Times Roman as basic font

\usepackage{color}

\makeatletter
\@addtoreset{equation}{section}
\makeatother

\newtheorem{theorem}{Theorem}[section]
\newtheorem{lemma}[theorem]{Lemma}
\newtheorem{proposition}[theorem]{Proposition}
\newtheorem{corollary}[theorem]{Corollary}

\newtheorem{remark}[theorem]{Remark}

\renewcommand{\a}{\alpha}

\newcommand{\Si}{\Sigma}
\newcommand{\om}{\omega}

\newcommand{\Om}{\Omega}

\newcommand{\R}{\mathbb{R}}

\newcommand{\N}{\mathbb{N}}
\newcommand{\Z}{\mathbb{Z}}

\newcommand{\T}{\mathbb{T}}

\newcommand{\mf}[1]{{\mathfrak #1}}
\newcommand{\mb}[1]{{\mathbf #1}}
\newcommand{\bb}[1]{{\mathbb #1}}
\newcommand{\bs}[1]{{\boldsymbol #1}}
\newcommand{\ms}[1]{{\mathscr #1}}

\definecolor{bblue}{rgb}{.2,0.2,.8}

\begin{document}

\title[Viscous Burgers equations]{Derivation of viscous Burgers
  equations from weakly asymmetric exclusion processes}
\author{M. Jara, C. Landim and K. Tsunoda}

\address{\noindent IMPA, Estrada Dona Castorina 110, CEP 22460 Rio de Janeiro, Brasil.
\newline e-mail: \rm
  \texttt{mjara@impa.br} }

\address{\noindent IMPA, Estrada Dona Castorina 110, CEP 22460 Rio de
  Janeiro, Brasil and CNRS UMR 6085, Universit\'e de Rouen, Avenue de
  l'Universit\'e, BP.12, Technop\^ole du Madril\-let, F76801
  Saint-\'Etienne-du-Rouvray, France.  \newline e-mail: \rm
  \texttt{landim@impa.br} }

\address{\noindent Department of Mathematics, Osaka University,
Osaka, 560-0043, Japan.
\newline e-mail:  \rm \texttt{k-tsunoda@math.sci.osaka-u.ac.jp}}



\begin{abstract}
  We consider weakly asymmetric exclusion processes whose initial
  density profile is a small perturbation of a constant. We show that
  in the diffusive time-scale, in all dimensions, the density defect
  evolves as the solution of a viscous Burgers equation.
\end{abstract}

\maketitle

\section{Introduction}
\label{intro}

One of the main open problems in nonequilibrium statistical mechanics
is the derivation of the hydrodynamic equations of fluids, the
so-called Euler and Navier-Stokes equations, from the microscopic
Hamiltonian dynamics.

In contrast with the Euler equations, the Navier-Stokes equations are
not scale invariant. They are obtained as corrections of the Euler
equations by adding a small viscosity, materialized as a second order
derivative of the conserved quantities.

Almost thirty years ago, Esposito, Marra and Yau \cite{emy1, emy2}
initiated the investigation of the time evolution of small
perturbations of the density profile around the hydrodynamic limit for
stochastic systems, deriving the incompressible limit for asymmetric
simple exclusion processes in dimension $d\ge 3$.

To describe their result, fix a scaling parameter $n\in\N$, and denote
by $\color{bblue} \T_n^d=(\Z/n\Z)^d$ the $d$-dimensional discrete
torus with $n^d$ points. Elements of $\bb T^d_n$ are represented by
the letters $x$, $y$, $z$. Denote the configuration space by
$\color{bblue} \Om_n=\{0,1\}^{\T_n^d}$ and by
$\color{bblue} \eta=\{\eta_x:x \in\T_n^d\}$ the elements of $\Om_n$,
which describes a configuration on $\T_n^d$ such that $\eta_x=1$ if
there is a particle at $x\in\T_n^d$ and $\eta_x=0$ otherwise.  For a
configuration $\eta\in\Om_n$, let $\sigma^{x,y}\eta$ be the
configuration of particles obtained from $\eta$ by exchanging the
occupation variables $\eta_x$ and $\eta_y$:
\begin{align*}
(\sigma^{x,y}\eta)_z \;=\; \begin{cases}
     \eta_y  & \text{if $z=x$}\;, \\
     \eta_x  & \text{if $z=y$}\;, \\
     \eta_z    & \text{otherwise}\;.
\end{cases}
\end{align*}

Consider the asymmetric exclusion process on $\Om_n$. This is the
Markov chain whose generator, denoted by $L_n^A$, applied to a
function $f: \Omega_n \to \bb R$ is given by
\begin{equation}
\label{int-1}
(L_n^A f)\, (\eta)\;=\; \sum_{x\in\T_n^d} \sum_{j=1}^d r_{x,j} (\eta) \, 
\{f(\sigma^{x,x+e_j}\eta)-f(\eta)\}\;,
\end{equation}
where $\color{bblue} \{e_j : 1\le j\le d\}$ represents the canonical
basis of $\bb R^d$,
$r_{x,j} (\eta) = p_j \, \eta_x \, (1-\eta_{x+e_j}) \,+\, q_j
\eta_{x+e_j} \, (1-\eta_{x})$ and $0\le p_j \le 1$, $q_j=1-p_j$.

Denote by $\bs \theta = (\theta_1, \dots, \theta_d)$ the points of the
$d$-dimensional continuous torus $\bb T^d=[0,1)^d$ and by
$\color{bblue} \nabla F$ the gradient of a function
$F: \bb T^d \to \bb R$,
$\nabla F = (\partial_{\theta_1} F, \dots, \partial_{\theta_d} F)$.
It is well known \cite{rez90,kl}, that in the hyperbolic scaling the density
profile evolves according to the inviscid Burger's equation
\begin{equation*}
\partial_t u  \;+\; \mb m \cdot \nabla\, \sigma_0 (u) \;=\; 0 \;,
\end{equation*}
where $\sigma_0(\alpha) = \alpha (1-\alpha)$ is the mobility and $\mb
m$ is the vector whose coordinates are given by $m_j = p_j - q_j$.

In dimension $d\ge 3$, the macroscopic current $\mb m\, \sigma_0 (u)$
is expected to have a correction of order $1/n$ and be given by
$\mb m\, \sigma_0 (u) - (1/n) \sum_k a_{j,k} (u) \partial_{\theta_k}
u$
for some diffusion coefficient $a$. If this is the case, the partial
differential equation which describes the evolution of the density
becomes
\begin{equation*}
\partial_t u  \;+\; \mb m \cdot \nabla\, \sigma_0 (u) \;=\; 
\frac 1n\, \sum_{j,k} \partial_{\theta_j} \big(\,
a_{j,k}(u)\, \partial_{\theta_k} u\, \big) \;.
\end{equation*}

If we start from a density which is a $(1/n)$-perturbation of the
constant profile equal to $1/2$,
$u_0(\bs \theta) = (1/2) + \epsilon_n\, v_0(\bs \theta)$, where
$\epsilon_n = 1/n$, if we rescale time by an extra factor $n$ and
assume that the density profile remains at all times a
$(1/n)$-perturbation of the constant profile equal to $1/2$,
$u(t,\bs \theta) = (1/2) + \epsilon_n\, v(t,\bs \theta)$, as
$\sigma_0'(1/2)=0$, a Taylor expansion yields that the perturbation $v$
is expected to solve the viscous Burgers equation
\begin{equation}
\label{int-2}
\partial_t v  \;=\; \mb m \cdot \nabla\, v^2 \;+\; 
\sum_{j,k} a_{j,k}(1/2)\, \partial^2_{\theta_j, \theta_k} v \;.
\end{equation}
This is the content of the main result of Esposito, Marra and Yau
\cite{emy1, emy2} which we now state.  Note that one can consider a
perturbation around a general constant profile $\a\in(0,1)$ by
performing a Galilean transformation [see Remark \ref{rm6}].

Recall that a function $f: \{0,1\}^{\bb Z^d} \to \bb R$ is said to be
a local function or a {\color{bblue} cylinder function} if it depends
on the configuration $\eta$ only through a finite number of
coordinates.

Denote by $\color{bblue} \{\tau_x : x\in \bb Z^d\}$ the group of
translations acting on $\Omega_n$: For a configuration
$\eta\in \Omega_n$, $\tau_x \eta$ is the configuration given by
$(\tau_x \eta)_z = \eta_{x+z}$, where the sum is taken modulo $n$. We
extend the translations to functions $f: \Omega_n \to \bb R$ by
setting $(\tau_x f)(\eta) = f(\tau_x \eta)$, $x\in \bb Z^d$,
$\eta\in \Omega_n$.

Let $\color{bblue} \nu_\alpha$, $0\le \alpha\le 1$, be the product
measure on $\{0,1\}^{\bb Z^d}$ with density $\alpha$. For a continuous
function $u:\T^d\to[0,1]$, denote by $\nu^n_{u(\cdot)}$ the Bernoulli
product measure on $\Om_n$ with marginal density $u(x/n)$:
\begin{equation}
\label{23}
\nu^n_{u(\cdot)}\{\eta(x)=1\}\;=\;u(x/n)\;, \quad x\in\T_n^d.
\end{equation}
Fix a density $v_0 : \bb T^d \to \bb R$, and let $\nu_t^n$, $t\ge 0$,
be the measure $\nu^n_{(1/2) + \epsilon_n v(t, \cdot)}$, where
$v(t, \bs \theta)$ is the solution of equation \eqref{int-2} with
initial condition $v_0$.

Denote by $\eta^n(t)$ the Markov chain on $\Omega_n$ induced by the
generator $n^2\, L^A_n$, where $L^A_n$ has been introduced in
\eqref{int-1}. Note that time has been rescaled diffusively. For a
probability measure $\mu$ on $\Omega_n$, denote by $\bb P_\mu$ the
distribution of the process $\eta^n(t)$ starting from
$\mu$. Expectation with respect to $\bb P_\mu$ is represented by
$\bb E_\mu$.

Fix a smooth density profile $v_0: \bb T^d \to \bb R$, and distribute
particles on $\bb T^d_n$ according to
$\nu^n_0 = \nu^n_{(1/2) + \epsilon_n v_0(\cdot)}$.  Then, in dimension
$d\ge 3$, for every $t>0$, continuous function $G: \bb T^d \to \bb R$,
and cylinder function $\Psi: \{0,1\}^{\bb Z^d} \to \bb R$,
\begin{equation}
\label{int-6}
\lim_{n\to\infty} \bb E_{\nu^n_0}
\Big[\, \dfrac{1}{n^{d-1}} \, \Big| \sum_{x\in\T_n^d} G(x/n)\,
\Big\{ (\tau_x\Psi) \big(\eta^n(t) \big) \,-\,
E_{\nu_{\rho_n(t,x)}}[\Psi] \, \Big\}\,\Big|\,  \Big]
\;=\; 0\;,
\end{equation}
where  $\rho_n(t,x) = (1/2) + \epsilon_n v(t,x/n)$ and, recall,
$\nu_\alpha$ stands for the Bernoulli product measure with density
$\alpha$.

The proof of this result is based on a sharp estimate of the relative entropy.
Let $\Si_n$ be the set of all probability measures on $\Om_n$.  For a
reference measure $\nu\in\Si_n$, define the relative entropy
$H_n(\cdot \,|\,\nu)$ with respect to $\nu$ by
\begin{equation*}
H_n(\mu \,|\, \nu)\;=\; \sup_f \Big\{ \int_{\Omega_n} f\, d\mu \,-\,
\log \int_{\Omega_n}  e^f \, d\nu \Big\}\;,
\end{equation*}
where the supremum is carried over all functions $f: \Om_n \to \bb R$.
It is well known that
\begin{equation}
\label{01}
H_n(\mu|\nu)\;=\;
\int_{\Om_n} \dfrac{d\mu}{d\nu}\, \log{\dfrac{d\mu}{d\nu}} \, d\nu \;,
\end{equation}
if $\mu$ is absolutely continuous with respect to $\nu$, while
$H_n(\mu|\nu) = \infty$ if this is not the case.

Denote by $\{S^n_t : t \ge 0\}$ the semigroup of the Markov chain
$\eta^n_t$ rescaled diffusively. Hence, $\mu^n S^n_t$ represents the
state of the process at time $t$ provided the initial state is
$\mu^n$. Esposito, Marra and Yau \cite{emy1, emy2} proved that in
dimension $d\ge 3$,
\begin{equation*}
\lim_{n\to \infty} \frac 1{n^{d-2}}\,  H_n(\mu^nS_t^n|\nu_t^n)
\;=\;  0 \;,
\end{equation*}
where $\nu^n_t$ has been introduced just below \eqref{23}.  It is not
difficult to deduce \eqref{int-6} from the previous bound.

The result is restricted to $d\ge 3$, as in dimension $1$ and $2$
Gaussian fluctuations of order $n^{-d/2}$ appear around the
hydrodynamic limit and $n^{-d/2}$ is at least of the order of $1/n$ in
dimensions $1$ and $2$. For more details, see \cite[Section 1]{emy1}.\smallskip

In this article, we pursue the investigation of the time evolution in
the hydrodynamic limit of densities in the vicinity of constant
profiles by considering weakly asymmetric exclusion processes. These
are Markov processes on $\Omega_n$ whose generator $L_n$ acts on
cylinder functions as $L_n f = n^2 L^S_n f + n L^T_n f$, where $L_n^S$
represents the generator of the speed-change, symmetric exclusion
process given by
\begin{equation}
\label{int-7}
(L_n^Sf)\, (\eta)\;=\; \sum_{x\in\T_n^d} \sum_{j=1}^d 
c_j (\tau_x \eta) \, \{f(\sigma^{x,x+e_j}\eta)-f(\eta)\}\;,
\end{equation}
and $L_n^T$ the generator of the speed-change totally asymmetric
exclusion process given by
\begin{equation}
\label{11}
(L_n^Tf)\, (\eta)\;=\; \sum_{x\in\T_n^d} \sum_{j=1}^d \mb m_j \,
c_j(\tau_x \eta) \, \eta_x \, (1-\eta_{x+e_j}) \, 
\{f(\sigma^{x,x+e_j}\eta)-f(\eta)\}\;.
\end{equation}
In this formula, $c_j: \{0,1\}^{\bb Z^d} \to \bb R$, $1\le j\le d$,
are cylinder functions and $\mb m = (\mb m_1, \dots, \mb m_d)$ is a fixed
vector in $\bb R^d$. In this paper, we assume that $c_j$ does not
depend on the occupation variables $\eta_0$ and $\eta_{e_j}$ and
satisfies the gradient conditions \eqref{02}. Under these conditions,
one can see that the generator $L_n$ is invariant with respect to the
Bernoulli measures.

Note that the symmetric generator has been speeded-up by $n^2$, while
the asymmetric one by $n$. In other words, we consider a weakly
asymmetric system in a diffusive time scale $n^2$ with asymmetry
strength of order $1/n$.

The hydrodynamic equation of the weakly asymmetric speed-change
exclusion process is given by
\begin{equation*}
\partial_t u  \;=\; \nabla \cdot [ \, D(u) \, \nabla u \, ] 
\;-\; \nabla \cdot [\, \sigma (u) \, \mb m\,]  \;,
\end{equation*}
where the matrices $D(\cdot)$ and $\sigma (\cdot)$ represent the
diffusivity and the mobility, respectively. By further accelerating
the symmetric part of the dynamics by $b_n$, the asymmetric one by
$a_n$, and by assuming that the density is an
$\epsilon_n$-perturbation of a constant $\alpha$, viz.
$u(t,\theta) = \alpha + \epsilon_n v(t,\theta)$, we get from the
previous equation that
\begin{align*}
\partial_t v \; & =\;
b_n \, \nabla \cdot \big[ \, D (\alpha) \, \nabla v \,\big]
\;+\; b_n \, \epsilon_n\, \nabla \cdot \big[ \,  v \,  D' (\alpha)
\, \nabla v \,\big] \\
\; & -\; a_n \nabla\cdot[\,  v \, \sigma' (\alpha)\, \mb m \, ]
\; -\; (1/2)\,  a_n\, \epsilon_n\,
\nabla\cdot[\,  v^2 \, \sigma'' (\alpha)\, \mb m \, ] \;.
\end{align*}
There are many ways to handle the right-hand side. One of them is to
set $b_n=1$, $a_n = \epsilon^{-1}_n$, and assume that
$\sigma' (\alpha)=0$. In this case, up to smaller order terms, the
equation becomes
\begin{equation}
\label{int-3}
\partial_t v \; =\;
\nabla \cdot \big[ \, D (\alpha) \, \nabla v \,\big]
\; - \; (1/2)\,   \nabla\cdot[\,  v^2 \, \sigma'' (\alpha)\, \mb m \, ] \;.
\end{equation}

Assume, therefore, that $\sigma' (\alpha)=0$ for some
$\alpha\in (0,1)$. Note that this $\a$ always exists since
each entry of $\sigma$ is smooth and vanishes at $0$ and $1$.
Consider the weakly asymmetric exclusion
process in which the asymmetric part of the generator has been
speeded-up by $a_n n$ [instead of $n$] for some sequence
$a_n \to \infty$ and $a_nn^{-1}\to0$. Note that the latter condition ensures that
the operator $n^2[L_n^S+(a_n/n)L_n^T]$ becomes a Markovian generator
for sufficiently large $n$.  Denote by $v=v(t,\theta)$ the solution of
\eqref{int-3} with a smooth initial condition
$v_0: \bb T^d \to \bb R$.  Distribute particles on $\bb T^d_n$
according to $\nu^n_0 = \nu^n_{\alpha + \epsilon_n v_0(\cdot)}$, where
$\epsilon_n = 1/a_n$.  The first main result of this article states
that under some hypotheses on $a_n$, for every $t>0$, continuous
function $G: \bb T^d \to \bb R$, and cylinder function
$\Psi: \{0,1\}^{\bb Z^d} \to \bb R$,
\begin{equation}
\label{int-5}
\lim_{n\to\infty} \bb E_{\nu^n_0}
\Big[\, \dfrac{1}{n^{d}\epsilon_n} \, \Big| \sum_{x\in\T_n^d} G(x/n)\,
\Big\{ (\tau_x\Psi) \big(\eta^n(t) \big) \,-\,
E_{\nu_{\rho_n(t,x)}}[\Psi] \, \Big\}\,\Big|\,  \Big]
\;=\; 0\;,
\end{equation}
where  $\rho_n(t,x) = \alpha + \epsilon_n v(t,x/n)$.

As above, the proof of this result is based on an estimate of the
relative entropy of the state of the process with respect to a product
measure. We start the presentation of this bound with a remark which
elucidates what is needed. In Lemma \ref{as3} below, we show that in
order to single out an $\epsilon_n$-perturbation of the density around
a constant profile we need the entropy of the state of the process
with respect to the inhomogeneous product measure associated to the
density profile $\alpha + \epsilon_n v(t,x/n)$ to be of an order much
smaller than $n^d \, \epsilon^2_n$.

To state the entropy bound, denote by $d$ the dimension, and let
$(g_d(n) : n\ge 1)$ be the sequences given by
\begin{align}
\label{int-4}
g_d(n)\;=\;
\begin{cases}
n\,, & \text{if $d=1$\,,} \\
\log n\,, & \text{if $d=2$\,,} \\
1\,, & \text{for $d\ge 3$\,.}
\end{cases}
\end{align}
Following Jara and Menezes in \cite{jm2}, we prove in Theorem
\ref{main1} that under certain assumptions on the initial profile
$v_0$, the sequence $a_n$ and the initial distribution of particles,
for all $t>0$ there exists a finite constant $C=C(t)$, such that
\begin{equation*}
H_n(\mu^nS_t^n|\nu_t^n)\;\le \; C \, n^{d-2}\, g_d(n) \;,
\end{equation*}
where $\nu_t^n$ stands for the inhomogeneous product measure
associated to the density profile $\alpha + \epsilon_n v(t,x/n)$. 
This entropy estimate and a simple argument, presented in the proof of
Corollary \ref{main2}, yield \eqref{int-5}. Lemma \ref{as3} and
\eqref{int-4} yield some restrictions on $\epsilon_n$ discussed in
Remark \ref{rm4} below. \smallskip

We here mention related results, which establish the incompressible
limits for interacting particle systems: Esposito, Marra and Yau
\cite{emy1, emy2}, Quastel and Yau \cite{qy98}, Beltr\'an and Landim
\cite{bl}.  We also mention recent results, which study the entropy
estimate as in Theorem \ref{main1}.  The entropy estimate as in
Theorem \ref{main1} has been established in Jara and Menezes
\cite{jm1, jm2} to study the nonequilibrium fluctuations for
interacting particle systems. By establishing a similar entropy
estimate, Funaki and Tsunoda \cite{ft18} derived the motion by mean
curvature from Glauber-Kawasaki processes.

We conclude this introduction mentioning two other ways to detect
the evolution of small perturbations around the hydrodynamic limit.
Dobrushin \cite{d}, Dobrushin, Pellegrinotti, Suhov and Triolo
\cite{dpst88}, Dobrushin, Pellegrinotti, Suhov \cite{dps90} and
Landim, Olla, Yau \cite{loy96, loy97} investigated the first order
correction to the hydrodynamic equation. Landim, Valle and Sued \cite{lsv04}
examined the evolution of the density profile in the orthogonal
direction to the drift when the initial condition is constant along
the drift direction. Versions of these results might be problems for
future investigation.

\section{Notation and results}
\label{setting}

\subsection{Model}

Recall that we denote by $\{ e_j : j=1,\dots, d\}$ the canonical basis
of $\R^d$. Fix cylinder functions $c_j:\{0,1\}^{\bb Z^d} \to \bb R_+$,
$1\le j\le d$.  Assume that $c_j$ does not depend on $\eta_0$,
$\eta_{e_j}$ and that the gradient conditions are in force: For each
$j$, there exist cylinder functions $g_{j,p}$ and finitely-supported
signed measures $m_{j,p}$, $1\le p\le n_j$, such that
\begin{equation}
\label{02}
c_j (\eta) \, [\, \eta_{0} \,-\, \eta_{e_j}\,] \,=\, 
\sum_{p=1}^{n_j} \sum_{y\in\bb Z^d}  m_{j,p}(y) \, (\tau_y \, g_{j,p}
)(\eta)\;, \quad \sum_{y\in\bb Z^d}  m_{j,p}(y) \;=\; 0\; .
\end{equation}
Denote by $\ell_0$ the size of the support of the measures
$m_{j,p}$. This is the smallest integer such that
\begin{equation*}
m_{j,p} (y) \;=\; 0 \quad \text{if}\;\; y\;\not\in\; \Lambda_{\ell_0}
\;:=\; \{-\ell_0, \dots , \ell_0\}^d \;.
\end{equation*}

Let $L_n^S$ be the generator of the speed-change exclusion process in
$\Om_n$ introduced in \eqref{int-7}, and let $L_n^T$ be the generator
of the speed-change totally asymmetric exclusion process in $\Om_n$,
introduced in \eqref{11}.

Recall that we denote by $\color{bblue} \nu_\alpha = \nu^n_\alpha$,
$0\le \alpha \le 1$, the Bernoulli product measure on $\Om_n$ or on
$\{0,1\}^{\bb Z^d}$ with density $\alpha$. 
Since we assume that $c_j$ does not depend on $\eta_0, \eta_{e_j}$,
for any $\a$, $L_n^S$ is reversible with respect to $\nu_\a$.
Moreover, this assumption together with the gradient conditions \eqref{02} ensures that
$L_n^T$ is invariant with respect to $\nu_\a$. For a cylinder function
$g:\{0,1\}^{\bb Z^d} \to \bb R$, let $\widetilde g: [0,1] \to \bb R$ be
the polynomial function given by
\begin{equation}
\label{28}
\widetilde g (\alpha) \;=\; E_{\nu_\alpha}[\,g\,]\;, \qquad
\alpha\,\in\, [0,1]\; .
\end{equation}
\smallskip

Denote by $D(\rho) = (D_{j,k}(\rho))_{1 \le j,k \le d}$, the
\emph{diffusivity} of the exclusion process, the matrix whose entries
are given by
\begin{equation}
\label{27}
D_{j,k}(\rho) \;=\; \sum_{p=1}^{n_j} D_p(j,k) \, \widetilde
g'_{j,p}(\rho)\;, \quad\text{where}\quad
D_p(j,k) \;=\; -\, \sum_{y} y_k\, m_{j,p}(y)\;.
\end{equation}
In this formula, $\widetilde g_{j,p}'$ represents the derivative of
the function $\widetilde g_{j,p}$. This later one is obtained through
equation \eqref{28} from the cylinder functions $g_{j,p}$ introduced
in \eqref{02}. We prove in Proposition \ref{l12} that $D(\rho)$ is a
diagonal matrix:
\begin{equation}
\label{32}
\sum_{p=1}^{n_j} D_p(j, k) \, 
\widetilde{g}'_{j,p}(\rho) \;=\; 0 
\quad\text{for}\quad k \,\not =\, j \;.
\end{equation}

Denote by $\sigma(\rho) = (\sigma_{i,j}(\rho))_{1 \le i,j\le d}$ the
\emph{mobility}, the diagonal matrix whose entries are given by
\begin{equation}
\label{33}
\sigma_{j,j}(\rho) \;=\; \rho\, (1-\rho) \, \widetilde{c}_j(\rho)\;. 
\end{equation}
We prove in Proposition \ref{l12} the Einstein relation, which in the
present context reads that for every $\rho \in (0,1)$, $1\le j\le d$,
\begin{equation}
\label{29}
\widetilde{c}_j(\rho) \;=\; \sum_{p=1}^{n_j} D_p(j, j) \, 
\widetilde{g}'_{j,p}(\rho) \quad\text{so that}\quad
\frac 1{\chi(\rho)} \; \sigma(\rho) \;=\; D(\rho)\;,
\end{equation}
where $\chi(\rho) = \rho\, (1-\rho)$ is the \emph{static
  compressibility}.  \smallskip

Recall that we denote by $\bb T^d = [0,1)^d$ the $d$-dimensional torus
and by the symbol $\theta = (\theta_1, \dots, \theta_d)$ elements of
$\bb T^d$. For a smooth function $u: \bb T^d \to \bb R$, let
$\color{bblue} \partial_{\theta_j} u$ be the partial derivative of $u$
in the $j$-th direction and let
$\color{bblue} \nabla u = (\partial_{\theta_1} u,
\dots, \partial_{\theta_d} u)$
be the gradient of $u$. Similarly, for a smooth vector field
$b = (b_1, \dots, b_d): \bb T^d \to \bb R^d$, denote by
$\nabla\cdot b$ its divergence:
$\color{bblue} \nabla\cdot b = \sum_j \partial_{\theta_j} b_j$.

Fix a sequence $(a_n : n \ge 1)$ such that $a_n\uparrow\infty$, and
let $\color{bblue} \epsilon_n = 1/a_n$. Denote by
$\{\eta^n(t):t\ge0\}$ the Markov process on $\Om_n$ generated by
the operator
\begin{equation*}
L_n \;=\; n^2 \left[ \, L_n^S \;+\; \dfrac{a_n}{n} \, L_n^T \right]\;.
\end{equation*}
As mentioned in the introduction, throughout the paper we assume $a_nn^{-1}\to0$,
and this condition ensures that
the operator $n^2[L_n^S+(a_n/n)L_n^T]$ becomes a Markovian generator
for sufficiently large $n$.
If $a_n$ is constant in $n$, then the process is a weakly asymmetric
speed-change exclusion process.  Therefore, formally, the hydrodynamic equation is
given by
\begin{equation}
\label{hdleq}
\partial_tu  \;=\; \nabla \cdot [ \, D(u) \, \nabla u \, ] 
\;-\; a_n \nabla\cdot[\, \sigma (u)\,  \mb m\, ]\;.
\end{equation}
Assume that there exists $\alpha_0\in (0,1)$ such that
\begin{equation}
\label{34}
\sigma'(\alpha_0) \;=\; 0 : \quad
\sigma_{j,j}'(\alpha_0) \;=\; 0 \;\; 
\text{for}\;\; 1\le j\le d \;.
\end{equation}
Assume, furthermore, that the initial condition $u^n_0$ is given by
$u_0^n= \alpha_0 + \epsilon_n\, v_0$, where $v_0:\T^d\to\R$ is a
smooth profile, and, recall, $\epsilon_n = 1/a_n$.  Write the solution
$u$ as $\alpha_0 + \epsilon_n\, v$. Since $\sigma'(\alpha_0)=0$, a
straightforward computation yields that, up to lower order terms,
$v:\T^d\times[0,\infty)\to\R$ is the solution of the Cauchy problem
\begin{align}
\label{beq}
\begin{cases}
\partial_tv \;=\;
\nabla \cdot [ \, D(\alpha_0) \, \nabla v \, ] 
\;-\; (1/2)\, \nabla\cdot[\,  v^2 \, \sigma'' (\alpha_0)\, \mb m \, ]
\;,\\ 
v(0,\cdot)=v_0 (\cdot) \;.
\end{cases}
\end{align}
From these observations, one might expect that the empirical measure of
the weakly asymmetric exclusion process suitably rescaled converges to
the solution of the viscous Burgers equation \eqref{beq}.
As mentioned in the introduction, one can consider a perturbation around a general constant
profile $\a\in(0,1)$ by performing a Galilean transformation [see Remark \ref{rm6}].

\subsection{Main results}
Let $u: \bb T^d \to [0,1]$ be a continuous function. Denote by $\Vert
u \Vert_\infty$ the supremum norm: $\color{bblue} \Vert u \Vert_\infty
= \sup_{\theta\in \bb T^d} | u(\theta)|$. 
Let $u_j: \bb T^d \to \bb R$, $j=1$, $2$, be two
continuous functions and let $u_j^n(\theta) = u(\theta) + \kappa_n
u_j(\theta)$, where $\lim_n \kappa_n =0$.
Assume that there exists $\delta>0$ such that $\delta \le u_j(\theta) \le 1-\delta$ for all
$\theta\in\bb T^d$, $j=1,2$.  The proof of the next
lemma relies on a simple Taylor expansion.

\begin{lemma}
\label{as3}
There exists a finite constant $C_0$, depending only on $\delta$ and
$\Vert u_1\Vert_\infty$, $\Vert u_2\Vert_\infty$, such that
\begin{equation*}
H_n \big(\nu^n_{u^n_2(\cdot)} \big |\nu^n_{u^n_1(\cdot)} \big)
\;=\; \frac{\kappa^2_n}{2}
\sum_{x\in\bb T^d_n} \frac{[u_2(x/n) - u_1(x/n)]^2}
{\chi(u(x/n))} \;+\; R_n\;,
\end{equation*}
where $|R_n|\le C_0 \, \kappa^3_n\, n^d$.
\end{lemma}

This result states that $H_n \big(\nu^n_{u^n_2(\cdot)} \big
|\nu^n_{u^n_1(\cdot)} \big)$ is of order $\kappa^2_n\, n^d$. In
particular, the density profile at the scale $\kappa_n$ of a
probability measure $\mu_n$ is not characterized if its relative
entropy with respect to $\nu^n_{u^n_1(\cdot)}$ is of order
$\kappa^2_n\, n^d$.

Denote by $\color{bblue} C^m(\bb T^d)$, $m\ge 1$, the set of $m$-times
continuously differentiable functions on $\bb T^d$, and by
$\color{bblue} C^{m+\beta}(\bb T^d)$, $0<\beta<1$, the set of
functions in $C^m(\bb T^d)$ whose $m$-th derivatives are
H\"older-continuous with exponent $\beta$.  Fix a function $v_0$ in
$C^3(\bb T^d)$. By \cite[Theorem V.6.1]{lsu1968}, for each $T>0$,
there exists a unique solution, represented by $v(t,x)$, of
\eqref{beq}. Denote by $\color{bblue} (S_t^n:t\ge0)$ the semigroup
associated to the generator $L_n$, and recall from \eqref{int-4} the
definition of the sequence $g_d(n)$.

\begin{theorem}
\label{main1}
Assume that $\epsilon_n\downarrow0$ and that
$n^2 \, \epsilon^4_n \le C_0\, g_d(n)$ for some finite constant
$C_0$. Recall hypothesis \eqref{34}. Suppose that $v_0$ belongs to
$C^{3+\beta}(\T^d)$ for some $0<\beta<1$. Let $v_t$ be the solution of
\eqref{beq}, $u^n_t = \alpha_0 + \epsilon_n v_t$ and
$\color{bblue} \nu_t^n = \nu_{u^n_t(\cdot)}^n$.  Consider a sequence
of probability measures $\{\mu^n:n\ge1\}$ on $\Om_n$ such that
\begin{equation*}
H_n(\mu^n|\nu_0^n)\;\le \; C_1 \, n^{d-2}\, g_d(n)\;,
\end{equation*}
for some finite constant $C_1$. Then, for every $T>0$, there exists a
finite constant $C_2 = C_2(T, v_0, C_0, C_1)$, such that for every $0\le t\le T$,
\begin{equation*}
H_n(\mu^nS_t^n|\nu_t^n)\;\le \; C_2 \, n^{d-2}\, g_d(n) \;.
\end{equation*}
\end{theorem}

The proof of this result is based on a two-blocks estimate due to Jara
and Menezes \cite{jm2} and stated below in Lemma \ref{l6}. 

For two sequences $(b_n : n\ge 1)$, $(c_n : n\ge 1)$ of non-negative
real numbers, we write $\color{bblue} b_n \ll c_n$ to mean that
$\lim_n b_n/c_n =0$.  In view of Lemma \ref{as3} and Theorem
\ref{main1}, to characterize the density profile at the scale
$\epsilon_n$, we need at least $n^{d-2}\, g_d(n) \ll
n^d\, \epsilon_n^2$. This is exactly the extra assumption of the next
corollary.

\begin{corollary}
\label{main2}
Besides the assumptions of Theorem \ref{main1}, assume that $
g_d(n) \ll n^2\, \epsilon_n^{2}$.  Then, for every $t\ge0$, every function $H$ in
$C^2(\T^d)$ and every cylinder function $\Psi:\{0,1\}^{\Z^d}\to\R$,
\begin{equation*}
\lim_{n\to\infty} E_{\mu^nS_t^n}\Big[\, 
\Big| \, \dfrac{1}{n^d\epsilon_n}\sum_{x\in\T_n^d} H(x/n)\,
(\tau_x\Psi)(\eta) \,-\,
\int_{\T^d} H(x) \, E_{\nu_t^n}[\Psi]\, dx\, \Big|\, \Big] \;=\; 0\,.
\end{equation*}
\end{corollary}

\begin{remark}
\label{rm4}
The conditions $g_d(n) \ll n^2\, \epsilon_n^2$ and $n^2\, \epsilon_n^4\le C_0 g_d(n)$
in Theorem \ref{main1} and Corollary \ref{main2} read as follows, respectively.
There exists a finite
constant $C_0$ such that
\begin{itemize}
\item[(a)] In dimension $1$, \vphantom{$\Big\{$}  $n^{-1/2} \ll
  \epsilon_n$ and $\epsilon_n \le C_0 n^{-1/4}$;
\item[(b)] In dimension $2$, \vphantom{$\Big\{$} $(\log n)^{1/2} \,
  n^{-1} \ll \epsilon_n$ and $\epsilon_n \le C_0 (\log n)^{1/4} \,
  n^{-1/2}$;
\item[(c)] In dimension $d\ge 3$, \vphantom{$\Big\{$} $n^{-1} \ll
  \epsilon_n$ and $\epsilon_n \le C_0 \, n^{-1/2}$.
\end{itemize}
\end{remark}

\begin{remark}
\label{rm1}
In all dimensions, in the scaling $\epsilon_n = n^{-d/2}$ one observes
the fluctuations of the density field. In dimension $1$, the condition
$n^{-1/2} \ll \epsilon_n$ is therefore optimal, while in dimension
$2$, there is an extra factor $(\log n)^{1/2}$. In dimension $d\ge 3$,
Esposito, Marra and Yau \cite{emy1, emy2} examined the incompressible
limit of the asymmetric simple exclusion process. They proved that a
perturbation of size $1/n$ of the density profile around a constant
evolves in the diffusive time-scale as the solution of \eqref{beq}.

In particular, we believe that to reach perturbations of size $1/n$ in
dimension $d\ge 3$ we have to improve Theorem \ref{main1} by adding
``non-gradient corrections'', that is, to add a local perturbation of
the state of the process, as it has been done in \cite{q92, v93,
  klo95} to derive the hydrodynamic behavior of non-gradient
interacting particle systems [cf. Chapter 7 of \cite{kl}].

The diffusive behavior of the asymmetric exclusion process has been
further investigated in \cite{loy97, lsv04}.
\end{remark}

\begin{remark}
\label{rm6}
Hypothesis \eqref{34} can be circumvented by performing a Galilean
transformation. Indeed, writing the solution of \eqref{hdleq} as
$\alpha_0 + \epsilon_n v(t, x-\epsilon^{-1}_n \sigma'(\alpha) \bs m
t)$,
we get, from a straightforward computation, that $v$ is the solution
of the Cauchy problem \eqref{beq}. This computation does not require
hypothesis \eqref{34}, as the higher order terms in $\epsilon_n$
cancel [one of them being $\nabla \cdot [v\, \sigma'(\alpha) \bs m]$].
\end{remark}

\begin{remark}
\label{rm3}
The assumption that $n^2 \, \epsilon^4_n \le C_0\, g_d(n)$ for some
finite constant $C_0$ is needed to estimate the linear terms of the
time-derivative of the relative entropy. This issue is further
discussed in Remarks \ref{rm5} and \ref{rm7} below.
\end{remark}

The paper is organized as follows. In Section \ref{sec2}, we compute
the time derivative of the entropy $H_n(\mu^nS_t^n|\nu_t^n)$. In
Section \ref{sec5}, we estimate the time derivative of the entropy and
we prove Theorem \ref{main1} and Corollary \ref{main2}. In Section
\ref{sec3}, we present the results on the viscous Burger's equation
\eqref{beq} needed in the proofs of the main results, and, in Section
\ref{sec4}, we compute the adjoint of the generator $L_n$ in 
$L^2(\nu^n_{u(\cdot)})$.

\section{Entropy production}
\label{sec2}

We estimate in this section the time derivative of the relative
entropy.  Fix $n\ge 1$, and recall that we denote by $(S_t^n : t\ge
0)$ the semigroup associated to the generator $L_n$. Fix a stationary
state $\nu_\alpha$, $0<\alpha<1$, and a probability measure $\mu$ on
$\Om_n$. Denote by $f_t$ the Radon-Nikodym derivative of $\mu S_t^n$
with respect to $\nu_\alpha$. An elementary computation yields that
\begin{equation*}
\frac d{dt} f_t \;=\; L^*_n\, f_t\;,
\end{equation*}
where $L^*_n$ stands for the adjoint of $L_n$ in $L^2(\nu_\alpha)$.

For a function $f:\Om_n\to\R$ and a probability measure $\nu$ on
$\Omega_n$, denote by $I(f;\nu)$ the Dirichlet form given by
\begin{equation}
\label{3}
I (f;\nu)=\sum_{x\in\T_n^d} \sum_{j=1}^d \int
\big\{\, \sqrt{f(\sigma^{x,x+e_j}\eta)} 
\,-\, \sqrt{f(\eta)}\, \big\}^2 \, \nu(d\eta)\;. 
\end{equation}

The proof of the next result, which is similar to the one of Lemma
6.1.4 in \cite{kl}, is left to the reader. Recall from \eqref{23} the
definition of the product measure $\nu^n_{u(\cdot)}$ associated to a
function $u: \bb T^d_n \to (0,1)$. For a function $w: \bb R_+ \times
\bb T^d_n \to (0,1)$, let $\color{bblue} \nu^n_{w(t)} = \nu^n_{w(t,
  \cdot)}$.

\begin{lemma}
\label{l1}
Fix $n\ge 1$ and $0<\alpha<1$.  Let $w: \bb R_+ \times \bb T^d_n \to
(0,1)$ be a differentiable function in time, and let $\mu$ be a
probability measure on $\Om_n$. Then,
\begin{equation*}
\frac{d}{dt} H_n(\mu\, S_t^n |\nu_{w(t)}^n) \; \le \;
-\, n^2\, I \big(g_t \,;\, \nu_{w(t)}^n \big)
\;+\; \int \Big\{ L^*_{w(t)} {\bf 1}
\,-\, \partial_t\log{\psi_t} \Big\} \, d\mu\, S_t^n \;,
\end{equation*}
where $g_t$ represents the Radon-Nikodym derivative of $\mu\, S_t^n$
with respect to $\nu_{w(t)}^n$, $g_t = d\mu\, S_t^n / d\nu_{w(t)}^n$,
$L^*_{w(t)}$ the adjoint operator of $L_n$ in $L^2(\nu_{w(t)}^n)$ and
$\psi_t$ the density given by $\psi_t=d\nu_{w(t)}^n/d\nu^n_\alpha$.
\end{lemma}

In view of the previous lemma, we need to compute the
integrand in the right hand side of the statement of the lemma.
To state the explicit formula of $L^*_{\varrho} {\bf 1} - \partial_t\log{\psi_t}$ for
a function $\varrho:\T_n^d\to(0,1)$, we need to introduce several notations.
This computation will be postponed to Section \ref{sec4}.

Consider a cylinder function $f: \{0,1\}^{\bb Z^d} \to \bb R$ and
a function $\varrho:\T_n^d\to(0,1)$.  Fix a
positive integer $n$ large enough for $\{-n/2, \dots, n/2\}^d$ to
contain the support of $f$. We also introduce the notion of Fourier
coefficients for local functions, c.f. \cite[Subsection 5.4]{klo}. For each $x\in \bb T^d_n$ and subset $B$
of $\bb T^d_n$, let
\begin{equation}
\label{04}
\mf f\, (x,\varnothing) \;:=\; E_{\nu_\varrho}[\tau_x f] \;,\quad 
\mf f\,  (x,B)  \;:=\; E_{\nu_\varrho} \big[ \, (\tau_x f) \, 
\xi_\varrho(B+x)\, \big] \;,
\end{equation}
where, for a subset $D$ of $\bb T^d_n$,
\begin{equation*}
D+x \;=\;  \{y+x : y\in D\} \;, \quad
\xi_\varrho(D) \;=\; \prod_{y\in D} 
[\, \eta(y) \,-\, \varrho(y)\,] \;.
\end{equation*}
When $D=\{x\}$ for some $x\in\T_n^d$, we shall denote $\xi_\varrho(D)$ by $\xi_\varrho(x)$
for simplicity.

Note that $\mf f\, (x,B) = 0$ if the set $B$ is not contained in the
support of $f$. More precisely, assume that $f$ depends on $\eta$ only
through $\{\eta(x) : x\in \Lambda_\ell\}$, where $\Lambda_\ell =
\{-\ell, \dots, \ell\}^d$. Then,
\begin{equation}
\label{05}
\mf f\, (x,B) \;=\; 0 \;\; \text{if $B$ is not a subset of $\Lambda_\ell$}\;.
\end{equation}

With these notations, we may write
\begin{equation}
\label{03}
(\tau_x f) (\eta) \;=\; \mf f\, (x,\varnothing) 
\;+\; \sum_{A} \mf f\, (x,A)\; \omega_\varrho (A+x)\;,
\end{equation}
where the sum is performed over all non-empty subsets $A$ of $\bb
T^d_n$ and
\begin{equation*}
\omega_\varrho(D) \;=\; \prod_{y\in D} 
\frac { \, \eta(y) \,-\, \varrho(y)} 
{ \varrho(y) \, [ 1 - \varrho(y)]}  \;\cdot
\end{equation*}
Note that $\{\omega_\varrho(D):D\subset\T_n^d\}$ forms an orthogonal basis of $L^2(\nu_\varrho^n)$. Denote by $\ms E_k = \ms E_{n,k}$ all subsets of $\bb T^d_n$ with $k$
elements: $\ms E_{k} = \{ A \subset \bb T^d_n : |A|=k\}$. A cylinder
function $\tau_x f$, $x\in \bb T^d_n$, is said to be of degree $k$ if
$\mf f(x,A)=0$ for all $A\not \in \ms E_{k}$.

In Section \ref{sec4}, we compute $L^*_{\varrho} {\bf 1}$
for a function $\varrho:\T_n^d\to[0,1]$ and the results are stated
in terms of coefficients $A_j(x), B_{j,p}^{(i)}(x)$, \dots, which are
defined there. Since we apply the results to
the function $w(t)$, to stress the dependence
we denote them with $t$ in the following paragraphs.
Moreover, as we shall consider the case $w(t) = (1/2) + \epsilon_n\, v^n(t)$
in the following subsection, we shall denote them with $n$.
For instance, $A_j(x)$ will be denoted by $A_j(t,x), A_j^n(t,x)$
when $\varrho=w(t), w(t) = (1/2) + \epsilon_n\, v^n(t),$ respectively.
Moreover, the contributions coming from the symmetric part or the asymmetric part
will be denoted with the superscript $i=1,2,$ respectively.

The explicit expression of $L^*_{w(t)} {\bf 1}$ requires some notation.
Some of the notation below is borrowed from Section \ref{sec4}.
Denote by $D_j$ the
difference operator defined by
\begin{equation}
\label{13}
(D_j F) (x) \;=\; F(x+e_j)\,-\, F(x)\;, \quad x\in \bb T^d_n\;.
\end{equation}
For $1\le j\le d$, $1\le p\le n_j$, $x\in \bb T^d_n$, $t\ge 0$
$A\subset \bb T_n^d$, let $\mf c_{j}(t,x,A)$, $\mf g_{j,p} (t,x,A)$ be
the Fourier coefficients, introduced in \eqref{04}, of the cylinder
functions $\tau_x c_j$, $\tau_x g_{j,p}$, respectively, with respect
to the measure $\nu_{w(t)}^n$:
\begin{align*}
& \mf c_{j}(t,x,A) \;=\; E_{\nu^n_{w(t)}}
\big[ \, (\tau_x c_j) \, \xi_{w(t)} (A+x)\, \big] \;, \\
&\quad
\mf g_{j,p} (t,x,A) \;=\; E_{\nu^n_{w(t)}}
\big[ \, (\tau_x g_{j,p}) \, \xi_{w(t)} (A+x)\, \big] \;,
\end{align*}
where $\xi_{w(t)} (B) = \prod_{x\in B} [\eta_x - w(t,x) ], B\subset\T_n^d$.

For $i=1$, $2$, let $A_j(t,x)$, $B^{(i)}_{j,p} (t,x)$,
$E^{(i)}_j (t, x)$, $F^{(i)}_j (t, x)$, $G^{(i)}_j (t, x)$,
$H^{(i)}_j (t, x, A)$, $I_j (t, x)$ be the functions obtained from
\eqref{09}, \eqref{15}, \eqref{16} by replacing $\varrho(x)$ by
$w(t,x)$.  For example,
\begin{equation*}
A_j(t,x) \;=\; \frac{\chi (w(t,x)) \,+\, \chi(w(t,x+e_j))}
{\chi (w(t,x)) \; \chi(w(t,x+e_j))}\;, 
\end{equation*}
\begin{equation*}
E^{(2)}_j (t, x) \;=\; -\, \frac {\mb m_j}2\, A_j(t,x)\, (D_j w_t) (x) \,
w(t,x) \, [1 - w(t,x+e_j)]\;.  
\end{equation*}
In the case of $H^{(i)}_j (t, x, A)$ one has also to replace the
Fourier coefficients $\mf c_{j}(x,A)$, $\mf g_{j,p}(x,A)$, computed
with respect to $\nu_\varrho$, by $\mf c_{j}(t,x,A)$,
$\mf g_{j,p}(t,x,A)$, respectively.

Let
\begin{align*}
H_j(t, x, A) \; =\; n^2\, H^{(1)}_j (t, x, A) \;+\;
a_n\, n\, H^{(2)}_j (t, x, A)\;.
\end{align*}
As mentioned before, note that the first or the second term in the
right-hand side comes from the contribution of the symmetric or the
asymmetric part of the generator, respectively.

The functions $H^{(1)}_j$, $H^{(2)}_j$, introduced in \eqref{41},
\eqref{42}, are defined in terms of the Fourier coefficients
$\mf c_j(t,x,A)$ and $\mf g_{j,p}(t,x,A)$. Since the functions $c_j$,
$g_{j,p}$ are cylinder functions, there exists $\ell\ge 1$ such that
$\mf c_j(t,x,A)=0$ and $\mf g_{j,p}(t,x,A)=0$ for all sets $A$ which
are not contained in $\{-\ell, \dots, \ell\}^d$ [cf. remark
\eqref{05}]. Hence, there exists $\ell\ge 1$ such that $H_j(t,x,A)=0$
if $A\not\subset \{-\ell, \dots, \ell\}^d$.

It also follows from the definitions of $H^{(1)}_j$, $H^{(2)}_j$,
given in \eqref{41}, \eqref{42}, that the functions of $x$ which
appear in the previous formula either contain the product of
derivatives [this is the case of $E_j$, $F_j$ and $G_j$] or a second
discrete derivative, which is the case of $B_{j,p}$ [see also the
paragraph before Lemma \ref{APl01}].

Denote by $j_{0,e_j}$ the instantaneous current over the bond
$(0,e_j)$. This is the rate at which a particle jumps from $0$ to
$e_j$ minus the rate at which it jumps from $e_j$ to $0$. It is given
by
\begin{equation}
\label{12}
j_{0,e_j} \;=\; c_j(\eta)\, [\, \eta_0 - \eta_{e_j}\,]\;. \quad
\text{Let}\;\;
j_{x, x+e_j} \;=\; \tau_x \, j_{0,e_j}\;, \quad x\in \bb T^d_n\;.
\end{equation}
The gradient conditions \eqref{02} assert that this current can be
written as a mean-zero average of translations of cylinder functions.

Next result is a consequence of Lemmata \ref{APl01} and \ref{APl02}.

\begin{lemma}
\label{l01}
We have that
\begin{equation*}
L_{w(t)}^{*} \, {\bf 1} \;=\; \sum_{j=1}^d
\sum_{x\in\T_n^d} K_j(t,x) \, \omega_x 
\; + \; \sum_{j=1}^d \sum_{A: |A|\ge 2}  \sum_{x\in\T_n^d} H_j(t,
  x, A) \, \omega (A+x)\;, 
\end{equation*}
where 
\begin{equation*}
K_j(t,x) \;=\; n^2\Big\{ E_{\nu^n_{w(t)}} \big[\, j_{x-e_j,x} \, \big] \,-\,  
E_{\nu^n_{w(t)}} \big[\, j_{x,x+e_j} \,  \big] \Big\}
\, -\, a_n\, n\, (\, D_j \, I_j \,)\, (t, x-e_j)\;,
\end{equation*}
the sum over $A$ is performed over finite subsets $A$ with at
least two elements, and
\begin{equation*}
\omega_x \;=\; \frac{\eta_x- w(t,x)}{\chi(w(t,x))}\;, \quad
\omega (B) \;=\; \prod_{x\in B} \omega_x \;, \quad
B\subset \bb Z^d\;. 
\end{equation*}
\end{lemma}

Note that $\omega_x$ depends on time, but this dependency is
frequently omitted from the notation to avoid long formulas. Also, to
stress the point at which it is evaluated, we write sometimes
$\omega (x)$ for $\omega_x$.

\begin{lemma}
\label{as4}
Under the assumptions of Lemma \ref{l1},
for every $t\ge 0$,
\begin{equation*}
\partial_t\log{\psi_t} \;=\;
\sum_{x\in\T_n^d} (\partial_t w) (t,x)\, \om_x\;. 
\end{equation*}
\end{lemma}

It follows from Lemmata \ref{l1}, \ref{l01}, \ref{as4}
that $L_{w(t)}^{*} \, {\bf 1} - \partial_t\log{\psi_t}$
presents only terms of degree $2$ or higher if $w(t,x)$ solves the
semi-discrete equation
\begin{equation*}
(\partial_t w)(t,x) \;=\; \sum_{j=1}^d K_j(t,x)\;.
\end{equation*}

\subsection{Perturbations of constant profiles}
We turn to the setting of Theorem \ref{main1}, and assume, without
loss of generality, that in hypothesis \eqref{34}, $\alpha_0 =1/2$.
Recall that $\epsilon_n = 1/a_n$ and assume, throughout this
subsection, that the function $w(t)$ of Lemma \ref{l1} is given by
$w(t) = (1/2) + \epsilon_n\, v^n(t)$ for some function
$v^n: \bb R_+ \times \bb T^d_n \to \bb R$. At this point we do not
suppose yet that $v^n(t)$ is the solution of \eqref{beq}.

Lemma \ref{l01} provides a formula for $L^*_{w(t)}{\bf 1}$. Many terms cancel
or simplify due to the special form of $w(t)$. In the next lemma we
present the result of these reductions.
As mentioned before, the coefficients $A_j^n(t,x), \dots$,
which will be defined below, formally coincide with
$A_j(t,x), \dots$, respectively.

Denote by $\nabla^n_j$ the discrete partial derivative in the $j$-th
direction. For a function $\varrho: \bb T^d_n \to \bb R$, $\nabla^n_j
\varrho$ is given by 
\begin{equation}
\label{24}
(\nabla^n_j \varrho) (x) \;=\;  n \, [\, \varrho (x+e_j) - \varrho
(x)\, ]\;, \quad  x\in\T_n^d\;.
\end{equation}
For $1\le j\le d$, $1\le p\le n_j$, $x\in \bb T^d_n$, let
\begin{equation}
\label{09b}
A^n_j(t,x) \;=\; \frac{\chi (w_t(x)) \,+\, \chi(w_t(x+e_j))}
{\chi (w_t(x)) \; \chi(w_t(x+e_j))}\;, 
\end{equation}
\begin{equation*}
C^n_j(t,x) \;=\; \mb m_j\, w_t(x) \, [\, 1 - w_t(x+e_j)\,]\;.
\end{equation*}

Let $U^{n,(1)}_{j} (t,x) = (\epsilon_n/n)\, (\nabla^n_j v^n)
(t,x), U^{n,(2)}_{j} (t,x)  \,=\,  -\, C^n_j(t,x)$ and
\begin{align*}
V_j^{n,(i)} (t,x) \;=\; [\, U_j^{n,(1)}(t,x) \,]^{3-i} \; [\, U_j^{n,(2)}(t,x) \,]^{i-1}\;,
\quad i=1,2\;.
\end{align*}
For $i=1$, $2$, $B^{n,(i)}_{j,p}, E^{n,(i)}_{j,p}, F^{n,(i)}_{j,p}, G^{n,(i)}_{j,p}, $ are defined as
\begin{align*}
B^{n,(i)}_{j,p} (t,x) \;&=\;  \frac 12\, \sum_{y\in\T_n^d} m_{j,p}(y) \,
A^n_j(t,x-y)\, \; U^{n,(i)}_{j} (t,x-y) \;, \\
E^{n,(i)}_j (t,x) \;&=\;  \frac 12\, 
A^n_j(t,x)\, V_j^{n,(i)} (t,x) \;, \\
F^{n,(i)}_j (t,x) \;&=\; -\, \frac{\epsilon_n}{2}\,
\frac {V_j^{n,(i)}(t,x)}{\chi(w_t \, (x+e_j))} \, 
\big\{ \, v^n \, (t,x) + v^n \, (t,x+e_j) \,\big\} \;, \\
G^{n,(i)}_j (t,x) \;&=\; -\, \frac{\epsilon_n}{2}\,
\frac {V_j^{n,(i)}(t,x)}{\chi(w_t \, (x))} \, 
\big\{ \, v^n \, (t,x) + v^n \, (t,x+e_j) \,\big\} \;,
\end{align*}
respectively. For $A\subset \bb T_n^d$ and $i=1,2$, let $J^{n,(i)}_j(t,x, A)$ be given
\begin{align*}
J^{n, (i)}_j(t,x, A) \; =\; 
& -\;  \Upsilon_{\{0,e_j\}} (A)\; 
V_j^{n,(i)} \, (t,x) \; \mf c_j(t,x,A\setminus \{0,e_j\}) \\
& +\; \Upsilon_{\{0\}} (A) \;
F^{n,(i)}_j \, (t,x) \; \mf c_j(t,x,A\setminus \{0\}) \\
& +\; \Upsilon_{\{e_j\}} (A) \;
G^{n,(i)}_j \, (t,x) \; \mf c_j(t,x,A\setminus \{e_j\}) \;,
\end{align*}
where, for two subsets $A$, $B$ of $\bb Z^d$,
\begin{equation}
\label{14}
\text{$\Upsilon_{B} (A) =1$ \; if $B\subset A$\,, \, 
and \, $\Upsilon_{B} (A)  =0$ \; otherwise}\;. 
\end{equation}
Here, the Fourier coefficients $\mf c_{j}(t,x,A)$,
$\mf g_{j,p} (t,x,A)$ are computed with respect to the product measure
$\nu^n_{w(t)}$.  Finally, let
\begin{equation*}
H^{n,(i)}_j(t, x, A) \; =\; E^{n,(i)}_j(t,x)\; \mf c_{j}(t,x,A)  
\;+\; \sum_{p=1}^{n_j} B^{n,(i)}_{j,p} (t,x) \; \mf g_{j,p}(t,x,A) \;+\;
J^{n,(i)}_j(t,x, A)\;,
\end{equation*}
\begin{align*}
H^n_j(t, x, A) \; =\; n^2\, H^{n,(1)}_j (t, x, A) \;+\;
a_n\, n\, H^{n,(2)}_j (t, x, A)\;.
\end{align*}

In the case where $w(t) = (1/2) + \epsilon_n v^n(t)$, Lemma
\ref{as4} and Lemma \ref{l01} become
\begin{equation}
\label{19}
\partial_t\log{\psi_t} \;=\; \epsilon_n\, 
\sum_{x\in\T_n^d} (\partial_t v^n) (t,x)\, \om_x\;. 
\end{equation}

\begin{lemma}
\label{l4}
Suppose that $w(t) = (1/2) + \epsilon_n v^n(t)$. Then,
\begin{equation*}
L_{w(t)}^{*} \, {\bf 1} \;=\; \sum_{j=1}^d
\sum_{x\in\T_n^d} K^n_j(t,x) \, \omega_x 
\; + \; \sum_{j=1}^d \sum_{A: |A|\ge 2}  \sum_{x\in\T_n^d} 
H^n_j(t,  x, A) \, \omega (A+x)\;, 
\end{equation*}
where 
\begin{equation*}
K^n_j(t,x) \;=\; n^2\Big\{ E_{\nu^n_{w(t)}} \big[\, j_{x-e_j,x} \, \big] \,-\,  
E_{\nu^n_{w(t)}} \big[\, j_{x,x+e_j} \,  \big] \Big\}
\, -\, a_n\, (\, \nabla^n_j \, I^n_j \,)\, (t, x-e_j)\;,
\end{equation*}
\begin{equation*}
\vphantom{\Big\{}
\text{$I^n_j(t,x)=  E_{\nu^n_{w(t)}}[\tau_x c_j] \, C^n_j(t,x)$, and}
\end{equation*}
$\om_x$ and $\om(B)$ are defined in Lemma \ref{l01}.
\end{lemma}

The next result is a consequence of
Lemmata \ref{l1}, \ref{as4}, \ref{l4}.

\begin{corollary}
\label{l8}
Suppose that $w(t) = (1/2) + \epsilon_n v^n(t)$.  All terms of degree
$1$ of $L_{w(t)}^{*} \, {\bf 1} - \partial_t\log{\psi_t}$ vanish as
long as $v^n(t,x)$ is the solution of the semi-discrete equation
\begin{equation}
\label{17}
(\partial_t v^n) (t,x) \;=\; a_n\, \sum_{j=1}^d K^n_j(t,x)\;,
\quad t\ge 0 \;,\; x\in\bb T^d_n\; .
\end{equation}
\end{corollary}

\begin{remark}
\label{rm5}
Note that the computation of $L_{w(t)}^{*} \, {\bf 1}$ for an
arbitrary profile $w(t): \bb T^d_n \to (0,1)$ reveals the
semi-discrete partial differential equation which describes the
macroscopic evolution of the density.

At this point, there are two possible choices. In Lemma \ref{l4}, we
may consider as reference state the product measure $\nu^n_{w(t)}$
whose density profile $w(t)$ is given by
$(1/2) + \epsilon_n\, v^n(t)$, where $v^n(t)$ is the solution of the
semi-discrete equation \eqref{17}, or the one given by
$(1/2) + \epsilon_n\, v(t)$, where $v(t)$ is the solution of the
semi-linear equation \eqref{beq}.

With the first choice, the terms of degree one in the expression
$L_{w(t)}^{*} \, {\bf 1} - \partial_t\log{\psi_t}$ vanish. To estimate
the terms of order $2$ or higher, uniform bounds of the discrete
derivatives of the solutions of the semi-discrete equation \eqref{17}
are needed.

With the second choice, the terms of degree one appear multiplied by a
small constant, but do not vanish and need to be estimated. In
contrast, the terms of degree $2$ or higher can be estimated with
bounds on the derivatives of the solutions of the semi-linear equation
\eqref{beq} provided by \cite{lsu1968}.
\end{remark}

We followed here the approach adopted by the previous authors and
sticked to the second choice.

\begin{remark}
\label{rm7}
The assumption that $n^2 \, \epsilon^4_n \le C_0\, g_d(n)$ for some
finite constant $C_0$ is needed to estimate the linear terms of the
time-derivative of the relative entropy [the linear terms of
$L^*_{w(t)} \bs 1 - \partial_t \log \psi_t$, computed in Lemmata
\ref{l01} and \ref{as4}]. Actually, equation \eqref{beq} is a
continuous version of the semi-discrete equation obtained by
considering the linear terms (in $\eta$) of the identity
\begin{equation}
\label{35}
L^*_{w(t)} \bs 1 \;-\;  \partial_t \log \psi_t \;=\; 0\;.
\end{equation}

One may try to weaken or remove the hypothesis $n^2 \,
\epsilon^4_n \le C_0\, g_d(n)$ by replacing equation \eqref{beq} by
the one obtained restricting \eqref{35} to the linear terms. In this
case, however, estimating the quadratic terms of \eqref{35} might be
more demanding. One may also try to weaken this hypothesis by adding
to equation \eqref{beq} terms of order $\epsilon^k_n$, $k\ge 2$. 
\end{remark}

\begin{remark}
\label{rm8}
In the case where $c_j(\eta)=1$, $\mb m_j=1$ for all $j$, the
semi-discrete equation \eqref{17} becomes
\begin{equation*}
(\partial_t v) (t,x) \;=\; \Big( 1 + \frac{a_n}{n}\Big)\, (\Delta_n v) (t,x) 
\;+\;  n \sum_{j=1}^d  \,  \big\{ 
v(t,x) \, v(t,x+e_j) \,-\, v(t,x-e_j) \, v(t,x) \big\}\;,  
\end{equation*}
where $\Delta_n \varrho$ stands for the discrete Laplacian:
\begin{equation*}
(\Delta_n \varrho) (x) \;=\; n^2 \sum_{j=1}^d
\big \{ \, \varrho (x+e_j) \,+\, \varrho (x-e_j) \,-\,  2 \,\varrho
(x)  \, \big\}\;.
\end{equation*}
\end{remark}

\section{Proof of Theorems  \ref{main1} and Corollary  \ref{main2}}
\label{sec5}

Assume, without loss of generality, that in hypothesis \eqref{34},
$\alpha_0 =1/2$. Assume, furthermore, that
$v: \bb R_+ \times \bb T^d \to \bb R$ is the solution of the
semi-linear equation \eqref{beq} and that
$w^n(t, x) = (1/2) + \epsilon_n \, v(t,x/n)$.  We refer
constantly to Section \ref{sec3} for properties of the solutions of
the viscous Burgers equation \eqref{beq}.

By Lemma \ref{l9}, for all $T>0$, there exists $\delta>0$ such that
\begin{equation}
\label{18}
\delta \;\le\; w^n(t, x) \;\le\; 1\,-\, \delta\;,
\end{equation}
for all $0\le t\le T$, $x\in\bb T^d_n$ and $n$ sufficiently
large. \smallskip

Let $L^n : \bb R_+ \times \bb T^d_n \to \bb R$ be given by
\begin{equation}
\label{20}
L^n (t,x) \;=\; \sum_{j=1}^d K^n_j (t,x) \,-\, \epsilon_n\,  (\partial_t
v)(t,x/n)\; . 
\end{equation}

\begin{lemma}
\label{l3}
Fix a density profile $v_0$ in $C^{3 + \beta}(\bb T^d)$ for some
$0<\beta<1$.  For every $T>0$, there exists a finite constant $C_0$,
depending only on $v_0$ and $T$, such that for all $0\le t\le T$,
$\gamma>0$,
\begin{equation*}
\int \sum_{x\in\T_n^d} L^n (t,x) \, \om_x \, d\mu\, S_t^n \;\le\; 
\frac 1\gamma\, H_n(\mu\, S_t^n |\nu_{w^n(t)}^n) \;+\; 
C_0 \, \gamma\, n^{d-2} \, (1 + n^2 \, \epsilon^4_n)\, 
e^{C_0 \, \gamma \, \kappa_n}\;,
\end{equation*}
where $\kappa_n = \epsilon^2_n + (1/n)$.
\end{lemma}

\begin{proof}
By the entropy inequality, the left-hand side is bounded by
\begin{equation*}
\frac 1{\gamma} \, H_n(\mu\, S_t^n |\nu_{w^n(t)}^n) \;+\;
\frac 1{\gamma} \, \log \int \exp\Big\{\gamma 
\sum_{x\in\T_n^d} L^n(t,x) \, \om_x \Big\} \, d \nu_{w^n(t)}^n\;,
\end{equation*}
for all $\gamma>0$. As $\nu_{w^n(t)}^n$ is a product measure, we may
move the sum outside the logarithm. Since $e^x \le 1 + x + (1/2) x^2
e^{|x|}$, $\log (1+a) \le a$, $a>0$, and since $\om_x$ has mean zero
with respect to $\nu_{w^n(t)}^n$, the second term of the previous
formula is bounded above by
\begin{equation*}
\frac {\gamma}{2} \, \sum_{x\in\T_n^d} \frac{L^n(t,x)^2}
{\chi(w^n(t,x))}  \, \exp\Big\{\gamma 
\, |L^n(t,x)| \, / \chi(w^n(t,x)) \Big\} \;,
\end{equation*}
because $E_{\nu_{w^n(t)}^n}[\omega^2_x] = 1/\chi(w^n(t,x))$. By Lemma
\ref{l5} and by \eqref{18}, the previous expression is bounded by 
\begin{equation*}
C_0 \, \gamma\,  n^{d-2} \,  (1 + n^2 \, \epsilon^4_n) \,
e^{C_0 \, \gamma \, [\epsilon^2_n + (1/n)]}\;,
\end{equation*}
for some finite constant $C_0$ which depends only on $v_0$ and $T$. This
completes the proof of the lemma.
\end{proof}

We turn to the quadratic or higher order term $H^n_j(t,x,A)$.  The
estimation is based on the following bound due to Jara and Menezes
\cite[Lemma 3.1]{jm2}.

\begin{proposition}
\label{l6} 
Fix a finite subset $A$ of $\bb Z^d$ with at least two elements.  For
every $\delta>0$, $a>0$ and $C_1<\infty$, there exists a finite
constant $C_0$, depending only on $\delta$, $A$, $C_1$ and $a$ such
that the following holds. For all $n\ge 1$, probability measures $\mu$
on $\Omega_n$, functions $u$, $J: \bb T^d_n \to \bb R$ such that
$\delta \le u(x) \le 1-\delta$ for all $x\in \bb T^d_n$, and
\begin{equation*}
\max_{x\in \bb T^d_n} \max_{1\le j\le d} 
\big|\, (\nabla^n_j u)(x)\, \big| \;\le\; C_1\;, \quad 
\max_{x\in \bb T^d_n} |J(x)| \;\le\; C_1\;,
\end{equation*}
we have that
\begin{equation*}
\int\sum_{x\in\T_n^d} J(x)\, \omega (A+x) \,  d\mu \;\le\;
a\, n^2\, I(g \,;\, \nu^n_{u(\cdot)}) \;+\; 
C_0 \big \{ \, H_n(\mu \,|\, \nu^n_{u(\cdot)} ) \,+\, 
n^{d-2}g_d(n) \big\}\;,
\end{equation*}
where 
\begin{equation*}
\omega_x \; =\; \frac{\eta_x - u(x)}{\chi(u(x))}\;, \quad \omega(B) \;=\;
\prod_{x\in B} \omega_x\;, \quad B\subset \bb Z^d\;, 
\end{equation*}
and $g = d\mu/d\nu^n_{u(\cdot)}$.
\end{proposition}

We show in the next paragraphs that the hypotheses of this proposition
are fulfilled for $u(x) = w^n(t,x)$, $J(x) = H^n_j(t,x,A)$. We first
prove the bounds for $u$ and then the ones for $J$.

By definition, $|(\nabla^n_j w^n)(t,x)| \le \epsilon_n\, \sup_{\theta\in
  \bb T^d} | (\partial_{\theta_j} v)(t,\theta)|$. Hence, by Lemma
\ref{l9}, for every $T>0$, there exists a finite constant $C_1=C_1(T,
v_0)$ such that for all $n\ge 1$,
\begin{equation}
\label{21}
\sup_{0\le t\le T} \max_{x\in \bb T^d_n} \max_{1\le j\le d} 
\big|\, (\nabla^n_j w^n)(t,x)\, \big| \;\le\; C_1\, \epsilon_n\;. 
\end{equation}
On the other hand, we have seen in \eqref{18} that for all $T>0$ there
exists $\delta>0$ such that $\delta \le w^n(t,x) \le 1-\delta$ for all
$x\in \bb T^d_n$, $0\le t\le T$ and $n$ sufficiently large.

The next lemma provides an estimate for the term $J(x) = H^n_j(t,x,A)$.

\begin{lemma}
\label{l11}
For each $T>0$, there exists a finite constant $C_0 = C_0(T,v_0)$
such that for all $n\ge 1$, 
\begin{equation*}
\sup_{0\le t\le T} \max_{x\in \bb T^d_n} \sup_{A\subset \bb Z^d}
\max_{1\le j\le d} 
\, \big|\, H^n_j(t,x,A) \, \big| \;\le\; C_0 \;,
\end{equation*}
where the supremum is carried over all finite subsets $A$ of $\bb
Z^d$. 
\end{lemma}

\begin{proof}
The proof is long, elementary and tedious. It follows from Lemma
\ref{l9} and from the definitions \eqref{09b} of the terms $A^n_j$,
$C^n_j$ that for each $T>0$, there exists a finite constant $C_0 =
C_0(T,v_0)$ such that for all $n\ge 1$,
\begin{equation*}
\sup_{0\le t\le T} \max_{x\in \bb T^d_n} \max_{1\le j\le d} 
\big|\, A^n_j(t,x) \, \big| \;\le\; C_0\;, \quad
\sup_{0\le t\le T} \max_{x\in \bb T^d_n} \max_{1\le j\le d} 
\big|\, C^n_j(t,x) \, \big| \;\le\; C_0\; . 
\end{equation*}
Furthermore, as $v(t,x)$ remains bounded in bounded time-intervals,
for each $T>0$, there exists a finite constant $C_0 = C_0(T,v_0)$
such that for all $n\ge 1$, 
\begin{equation*}
\sup_{0\le t\le T} \max_{x\in \bb T^d_n} \max_{|y|\le \ell_0}
\max_{1\le j\le d} 
n\, \big|\, A^n_j(t,x-y) \,-\,  A^n_j(t,x) \, \big| \;\le\; C_0\, \epsilon_n \;,
\end{equation*}
where $\ell_0$, introduced just after \eqref{02}, represents the size
of the support of the measures $m_{j,p}$.

Similar bounds hold for the functions $U^{n,(i)}_{j}$.  For each
$T>0$, there exists a finite constant $C_0 = C_0(T,v_0)$ such that for
all $n\ge 1$, $i=1$, $2$,
\begin{gather*}
\sup_{0\le t\le T} \max_{x\in \bb T^d_n} \max_{1\le j\le d} 
\max_{1\le p\le n_j} 
n^{2-i}\, \big|\, U^{n,(i)}_{j} (t,x) \, \big| \;\le\; 
C_0\, \epsilon_n^{2-i} \;, \\
\sup_{0\le t\le T} \max_{x\in \bb T^d_n} 
\max_{|y|\le \ell_0} \max_{1\le j\le d} \max_{1\le p\le n_j} n^{3-i}\,
\big|\, U^{n,(i)}_{j} (t,x-y) - U^{n,(i)}_{j} (t,x)\, \big| 
\;\le\; C_0\, \epsilon_n^{3-i} \; . 
\end{gather*}

It follows from the estimates on $A^n_j(t,x)$ and $U^{n,(i)}_{j}
(t,x)$ that for each $T>0$, there exists a finite constant $C_0 =
C_0(T,v_0)$ such that for all $n\ge 1$, $i=1$, $2$, 
\begin{equation*}
\sup_{0\le t\le T} \max_{x\in \bb T^d_n}  \max_{1\le j\le d} 
\max_{1\le p\le n_j} 
n^{3-i}\, \big|\, B^{n,(i)}_{j,p} (t,x) \, \big| \;\le\; 
C_0\, \epsilon_n \;.
\end{equation*}

Similarly, for each $T>0$, there exists a finite constant $C_0 =
C_0(T,v_0)$ such that for each $i=1$, $2$ and all $n\ge 1$,
\begin{align*}
& \sup_{0\le t\le T} \max_{x\in \bb T^d_n} \max_{1\le j\le d} 
n^{3-i} \, \big|\, E^{n,(i)}_j(t,x) \, \big| \;\le\; C_0\,
\epsilon_n^{3-i} \;, \\
&\quad
\sup_{0\le t\le T} \max_{x\in \bb T^d_n} \max_{1\le j\le d} 
n^{3-i} \, \big|\, F^{n,(i)}_j(t,x) \, \big| \;\le\; C_0\,
\epsilon_n^{4-i}\; , \\
&\qquad
\sup_{0\le t\le T} \max_{x\in \bb T^d_n} \max_{1\le j\le d} 
n^{3-i} \, \big|\, G^{n,(i)}_j(t,x) \, \big| \;\le\; C_0\,
\epsilon_n^{4-i}\; .
\end{align*}

Let $f$ be a cylinder function. Denote by $\mf f(t,x, A)$ the Fourier
coefficients of $f$ with respect to the measure $\nu_{w^n(t)}^n$, $w^n(t) =
(1/2) + \epsilon_n v(t)$.  It is clear, from the definition
\eqref{04}, that for all $n\ge 1$, $t\ge 0$, $x\in \bb T^d_n$,
$A\subset \bb Z^d$, 
\begin{equation}
\label{22}
\big|\, \mf f(t,x,A) \, \big| \;\le\; \Vert
f \Vert_\infty\;:=\; \sup_\eta \big|\, f(\eta) \, \big| \;.
\end{equation}

It follows from the previous estimate on the Fourier coefficients of
cylinder functions and from the bounds on $F^{n,(q)}_j$, $G^{n,(q)}_j$
that for each $T>0$, there exists a finite constant $C_0 = C_0(T,v_0)$
such that for each $i=1$, $2$, and all $n\ge 1$,
\begin{equation*}
\sup_{0\le t\le T} \max_{x\in \bb T^d_n} \sup_{A\subset \bb Z^d}
\max_{1\le j\le d} 
n^{3-i} \, \big|\, J^{n,(i)}_j(t,x,A) \, \big| \;\le\; C_0\,
\epsilon_n^{3-i} \;,
\end{equation*}
where the supremum is carried over all finite subsets $A$ of $\bb
Z^d$.

To complete the proof of the lemma, it remains to put together all
previous estimates.
\end{proof}

\begin{proof}
[Proof of Theorem \ref{main1}]
Let $\{\mu^n:n\ge1\}$ be a sequence of probability measures on $\Om_n$
satisfying the assumptions of the theorem.  Let $\mu_t^n=\mu^nS_t^n$
and $H_n(t)=H_n(\mu^n_t|\nu_t^n)$.

Lemma \ref{l1}, equation \eqref{19} and Lemma \ref{l4} provide a
formula for the derivative of $H_n(t)$.  Fix $T>0$. By \eqref{18},
there exists $\delta>0$ such that $\delta \le w^n(t,x) \le 1 - \delta$
for all $x\in \bb T^d_n$, $0\le t\le T$.  By \eqref{21},
\begin{equation*}
\kappa_T \,:=\; \sup_{0\le t\le T} \max_{x\in \bb T^d_n} \max_{1\le j\le d}
\big|\, (\nabla^n_j w^n)(t,x)\, \big| \;<\; \infty\;,
\end{equation*}
and by Lemma \ref{l11},
\begin{equation*}
H_T \,:=\; \sup_{0\le t\le T} \sup_{n\ge 0} \max_{A\subset \bb Z^d}
\max_{x\in \bb T^d_n} \max_{1\le j\le d}
\big|\, H^n_j (t,x,A)\, \big| \;<\; \infty\;.
\end{equation*}
Therefore, the hypotheses of Proposition \ref{l6} are in force for
$u(x) = w(t,x)$, $J(x) = H^n_j(t,x,A)$.

By hypothesis, $n^2\, \epsilon^4_n \le g_d(n)$. Hence, the second term
on the right-hand side of the statement of Lemma \ref{l3} is bounded
by $C_0 \, \gamma\, n^{d-2} \, g_d(n) \, \exp\{C_0\, \gamma\}$. In
particular, by Lemma \ref{l3} with $\gamma=1$ and by Proposition
\ref{l6} with $a=1/2$ applied to $\mu = \mu^n_t$, $u(x) = w^n(t,x)$,
$J(x) = H^n_j(t,x,A)$, there exists a finite constant $C_0$ such that
\begin{equation*}
H'_n(t) \;\le\; C_0\, H_n(t) \;+\; 
C_0 \,  n^{d-2} \, g_d(n) \;-\; 
\frac 12 \, n^2\, I(g^n_t \,;\, \nu_t^n) \;,
\end{equation*} 
where $g^n_t = d\mu^n_t/ d\nu_t^n$. At this point the assertion of the
theorem follows from Gronwall's lemma.
\end{proof}

\begin{proof}[Proof of Corollary \ref{main2}]
For simplicity, we prove the corollary in the case $\Psi(\eta)=\eta_0$.
Since $v_t$ is Lipschitz-continuous and $H$ is of class $C^2(\bb T^d)$,
\begin{equation*}
a_n\int_{\T^d}H(\theta) \, \Big\{ \dfrac{1}{2}
+\epsilon_nv(t,\theta) \Big\}\, d\theta\;=\;
\dfrac{a_n}{n^d}\sum_{x\in\T_n^d} H(x/n)\, 
\Big\{ \dfrac{1}{2}+\epsilon_nv(t, x/n) \Big\} 
\;+\;  O(\dfrac{a_n}{n})\; .
\end{equation*}
For each $x\in\T_n^d$, let $J^n_x(t)=H(x/n)(\eta_x^n(t)
-1/2-\epsilon_nv(t,x/n))$. Since $a_n/n\to 0$, to conclude the proof it
is enough to show that
\begin{equation*}
\lim_{n\to\infty} E_{\mu^nS_t^n} \Big[\, \Big|\,
\dfrac{a_n}{n^d} \sum_{x\in\T_n^d} J_x^n(t) \,\Big|\, \Big] 
\;=\;0\;.
\end{equation*}

By the entropy inequality and Theorem \ref{main1}, the expectation
appearing in the left-hand side can be bounded above by
\begin{equation*}
\dfrac{C_0}{K} \;+\; \dfrac{1}{K n^{d-2} g_d(n)}
\log E_{\nu_t^n} \Big[ \exp \Big\{ \, \Big| \dfrac{K a_n  g_d(n)}
{n^2} \sum_{x\in\T_n^d} J_x^n(t) \, \Big|\, \Big\} \, \Big]\;,
\end{equation*}
for all $K>0$ and some finite constant $C_0>0$.
Using $\exp \{|x|\} \le \exp\{x\} + \exp\{-x\}$,
it is enough to estimate the previous expression without the absolute value.
Indeed, the other term can be handled by the following argument similarly.
%and since $\limsup_n \gamma^{-1}_n \log
%(a_n + b_n)$ is bounded by the maximum between $\limsup_n
%\gamma^{-1}_n \log a_n$ and $\limsup_n \gamma^{-1}_n \log b_n$,
%provided $\gamma_n\to\infty$,
%it is enough to estimate the previous expression without the absolute value.

As $\nu_t^n$ is a product measure, the second term of the previous
displayed expression without the absolute value is equal to
\begin{equation*}
\dfrac{1}{K n^{d-2} g_d(n)} \sum_{x\in\T_n^d}
\log E_{\nu_t^n} \Big[ \exp \Big\{ \, \dfrac{K a_n  g_d(n)}
{n^2} J_x^n(t) \, \Big\} \, \Big]\;.
\end{equation*}
Since $\exp{x} \le 1+x+2^{-1}x^2\exp{|x|}$ and $\log{(1+y)}\le y$, as
$J_x^n(t)$ has mean zero with respect to $\nu_t^n$, the previous
displayed expression is bounded above by
\begin{equation*}
\dfrac{K\, a^2_n \, g_d(n)}{ 2\, n^{d+2} } \sum_{x\in\T_n^d}
E_{\nu_t^n} \big[ \, J_x^n(t)^2\, \big] \, \exp \Big\{ \, \dfrac{K a_n  g_d(n)}
{n^2} \| H\|_\infty \, \Big\} \;,
\end{equation*}
because $v_t$ is bounded.  Since $a^2_n \, g_d(n)/ n^{2} \to 0$, to
conclude the proof of the corollary, it remains to let $n\to\infty$
and then $K\to\infty$.
\end{proof}

\section{The Burgers viscous  equation}
\label{sec3}

We present in this section the properties of the solutions of the
Burgers viscous equation \eqref{beq} needed in the proof of Theorem
\ref{main1}. Without loss of generality, we assume that in hypothesis
\eqref{34}, $\alpha_0 =1/2$.

Recall the definition of the space $C^{m + \beta}(\bb T^d)$ introduced
just above Theorem \ref{main1}.  Fix a function $v_0$ in $C^{3 +
  \beta}(\bb T^d)$ for some $0<\beta<1$. According to \cite[Theorem
V.6.1]{lsu1968} there exists a unique solution to
\eqref{beq}. Moreover, the partial derivatives of the solution are
uniformly bounded on bounded time intervals. This later result is
summarized in the next lemma.

\begin{lemma}
\label{l9}
Assume that $v_0$ belongs to $C^{3 + \beta}(\bb T^d)$ for some
$0<\beta<1$.  For every $T>0$, there is a finite constant $C_0 =
C_0(T)$, depending only on $v_0$ and $T$, such that
\begin{equation*}
\sup_{0\le t\le T} \sup_{\theta\in \bb T^d} \,\big|\,
v(t,\theta)\,\big|\;\le\; C_0\;, \quad
\max_{1\le j\le d} \, \sup_{0\le t\le T} \, \sup_{\theta\in \bb T^d} \,\big|\,
(\partial_{\theta_j} v) (t,\theta)\,\big|\;\le\; C_0\;,
\end{equation*}
\begin{equation*}
\max_{1\le i, j\le d} \, \sup_{0\le t\le T} \, \sup_{\theta\in \bb T^d} \,\big|\,
(\partial^2_{\theta_i, \theta_j} v) (t,\theta)\,\big|\;\le\; C_0\;, 
\end{equation*}
\begin{equation*}
\max_{1\le i, j, k\le d} \, \sup_{0\le t\le T} \, \sup_{\theta\in \bb T^d} \,\big|\,
(\partial^3_{\theta_i, \theta_j, \theta_k} v) (t,\theta)\,\big|\;\le\; C_0\;. 
\end{equation*}
\end{lemma}

Recall the definition of the function $L_n : \bb R_+ \times \bb
T^d_n \to \bb R$ introduced in \eqref{20}.

\begin{lemma}
\label{l5}
Let $v: \bb R_+ \times \bb T^d \to \bb R$ be the solution of
\eqref{beq} and set $w(t,x) = (1/2) + \epsilon_n \, v(t,x/n)$,
$x\in\bb T^d_n$. Then, for every $T>0$, there is a finite constant
$C(T)$, depending only on $T$ and $v_0$, such that
\begin{equation*}
\sup_{0\le t\le T}\, \max_{x\in \bb T^d_n} 
\big|\, L^n (t,x) \, \big| \;\le\;  C(T)
\Big(\, \epsilon^2_n \,+\, \frac {1}{n} \, \Big)  \;,
\end{equation*}
for all $n\ge 1$.
\end{lemma}

The proof of this lemma is divided in several steps.

\begin{lemma}
\label{l13}
Fix $x\in \bb T^d_n$, $1\le j\le d$ and $0\le t\le T$. We claim that
\begin{equation*}
\begin{aligned}
& n^2\Big\{ E_{\nu^n_{w(t)}} \big[\, j_{x-e_j,x} \, \big] \,-\,  
E_{\nu^n_{w(t)}} \big[\, j_{x,x+e_j} \,  \big] \Big\} \\
&\qquad \;=\; \epsilon_n \,  D_{j,j} (1/2) \, 
(\partial^2_{x_j} v)\, (t,x/n) \; +\; \Big(\, \epsilon^2_n \,+\,
\frac {\epsilon_n}{n} \, \Big) \, R_n \;,
\end{aligned}
\end{equation*}
where $R_n$ is a remainder whose absolute value is bounded by
a finite constant $C(T)$ which depends only on $T$
and on $v$ through the $L^\infty$ norm of its first three derivatives.
\end{lemma}

\begin{proof}
By definition of the current and by assumption \eqref{02}, the
difference inside braces is equal to
\begin{equation}
\label{25}
\sum_{p=1}^{n_j} \sum_{y\in\bb Z^d} m_{j,p}(y) \,
E_{\nu^n_{w(t)}} \Big[\,  \tau_{x+y-e_j} \, g_{j,p}
- \tau_{x+y} \, g_{j,p} \, \Big]\;.
\end{equation}
We may rewrite the previous expectation as $E_{\nu^n_{w_{t,x}(\cdot)}} [\,
\tau_{y-e_j} \, g_{j,p} \, - \, \tau_y\, g_{j,p} \,]$, where $w_{t,x}
(z) = w(t,x+z)$, $z\in \bb T^d_n$. By Corollary \ref{l10}, this
expectation can be written as the sum of two expressions and a
remainder. We consider them separately.

The contribution to \eqref{25} of the first expression in Corollary
\ref{l10} is equal to
\begin{equation*}
\frac {-\, \epsilon_n}{n} \, 
\sum_{p=1}^{n_j} \sum_{y\in\bb Z^d} m_{j,p}(y) \,
\sum_{z} [\nabla^n_j v]\, (t, [x+z-e_j]/n)  \, 
E_{\nu^n_{w(t,x)}} \big[ \, \tau_y g_{j,p}  \, \omega_z \, \big] \;,
\end{equation*}
where $\nu^n_{w(t,x)}$ is the homogeneous product Bernoulli measure
with density $w(t,x)$. Fix $p$ and $y$. Performing a change of
variables we may rewrite the sum over $z$ as
\begin{equation}\label{1000}
\sum_{z} [\nabla^n_j v]\, (t, [x+z+y-e_j]/n)  \, 
E_{\nu^n_{w(t,x)}} \big[ \, g_{j,p}  \, \omega_z \, \big] \;.
\end{equation}
Performing a Taylor expansion around $(t,x/n)$,
\begin{align*}
[\nabla^n_{x} v](x') \;:&=\; n[\, v(t, [x+x']/n)  \,-\, v(t, x/n)\,] \\
\;&=\; \sum_{k=1}^dx_k'(\partial_{x_k} v)\, (t,x/n)\\
& \quad + \frac1n\, \sum_{k,k'} x_k'x_{k'}'(\partial_{x_k} v)\, (t,x/n)(\partial_{x_{k'}} v)\, (t,x/n) \;,
\end{align*}
plus $R_n/n^2$, where $R_n$ is a remainder whose absolute value is
bounded by $C_0$, for some constant $C_0$ depending only on $T$ and on
the $L^\infty$ norm of the first three derivatives of $v$. The
expression of the remainder $R_n$ may change below from line to line.
Note that $[\nabla^n_j v]\, (t, [x+z+y-e_j]/n) = [\nabla^n_{x} v](z+y) - [\nabla^n_{x} v](z+y-e_j)$.
Therefore an easy computation yields that the sum in \eqref{1000} becomes
\begin{align*}
& (\partial_{x_j} v)\, (t,x/n)  \, 
\sum_{z} E_{\nu^n_{w(t,x)}} \big[ \, g_{j,p}  \, \omega_z \, \big]  \\
&\quad \;-\; \frac {1}{2n} \, \sum_{z}  \Big\{ (\partial^2_{x_j}
v)\, (t,x/n)  \, -\, 2\, \sum_{k=1}^d (y_k+z_k)\, (\partial^2_{x_j, x_k} v)\,
(t,x/n) \Big\}\,
E_{\nu^n_{w(t,x)}} \big[ \, g_{j,p}  \, \omega_z \, \big] \;,
\end{align*}
plus $R_n/n^2$.

Since for each $j$ and $p$, $\sum_y m_{j,p}(y) = 0$, in view of
\eqref{26}, the contribution to \eqref{25} of the first expression in
Corollary \ref{l10} is equal to
\begin{align*}
& \frac {\epsilon_n}{n^2} \, \sum_{p=1}^{n_j} \sum_k D_p(j,k) \, 
(\partial^2_{x_j, x_k} v)\, (t,x/n) \, \widetilde{g}'_{j,p} (w(t,x))  
\; +\; \frac {\epsilon_n}{n^3} \, R_n  \\
&\quad 
=\; \frac {\epsilon_n}{n^2} \,  D_{j,j} (w(t,x)) \, 
(\partial^2_{x_j} v)\, (t,x/n) 
\; +\; \frac {\epsilon_n}{n^3} \, R_n \;,
\end{align*}
where $D_p(j,k)$, $D_{j,j}(\rho)$ have been introduced in
\eqref{27}. We used in the previous step the identities \eqref{32}.
As $w(t,x) = (1/2) \,+\, \epsilon_n \, v(t,x/n)$, by a Taylor
expansion, the previous expression is equal to
\begin{equation*}
\frac {\epsilon_n}{n^2} \,  D_{j,j} (1/2) \, 
(\partial^2_{x_j} v)\, (t,x/n) 
\; +\; \Big(\, \frac{\epsilon^2_n}{n^2} \,+\,
\frac {\epsilon_n}{n^3} \, \Big) \, R_n \;.
\end{equation*} 

We turn to the contribution to \eqref{25} of the second expression in
Corollary \ref{l10}. It is equal to
\begin{equation*}
\frac{\epsilon^{2}_n}{2\, n}\, \, 
\sum_{p=1}^{n_j} \sum_{y\in\bb Z^d} m_{j,p}(y) \,
\sum_{z\not = z'} c_{z,z'}  \, 
E_{\nu^n_{w(t,x)}} \big[ \, (\tau_y g_{j,p})  
\, \omega_z\, \omega_{z'} \, \big]  \;,
\end{equation*}
where 
\begin{align}\label{100}
c_{z,z'} \; & =\; \frac 1n\, (\nabla^n_j v)\, ([z-e_j]/n)  \,  
(\nabla^n_j v)\, ([z'-e_j]/n)  \notag\\
\; & - \; (\nabla^n_j v)\, ([z-e_j]/n)  \, [v(z'/n) - v(0)] \notag\\
\; & - \; (\nabla^n_j v) \, ([z'-e_j]/n)  \, [v(z/n) - v(0)] \;.
\end{align}
By a change of variables, we may write this expression as
 \begin{equation*}
\frac{\epsilon^{2}_n}{2\, n}\, \, 
\sum_{p=1}^{n_j} \sum_{y\in\bb Z^d} m_{j,p}(y) \,
\sum_{z\not = z'} c_{z+y,z'+y}  \, 
E_{\nu^n_{w(t,x)}} \big[ \, g_{j,p}
\, \omega_z\, \omega_{z'} \, \big]  \;.
\end{equation*}
The fact that $\sum_y m_{j,p}(y) =0$ yields
that this sum is equal to
\begin{equation*}
\frac{\epsilon^{2}_n}{2\, n}\, \, 
\sum_{p=1}^{n_j} \sum_{y\in\bb Z^d} m_{j,p}(y) \,
\sum_{z\not = z'} \left[ \,c_{z+y,z'+y} \, -\, c_{z,z'}  \, \right] 
E_{\nu^n_{w(t,x)}} \big[ \, g_{j,p}
\, \omega_z\, \omega_{z'} \, \big]  \;.
\end{equation*}
Note that $m_{j,p}(y)$ and the last expectation vanish except for a finite number
of $y, z, z'$. For such $y, z, z'$, a Taylor expansion shows that
$c_{z+y,z'+y} \, -\, c_{z,z'}$ is of order $n^{-2}$, uniformly in $y, z, z'$.
Therefore this sum is bounded in absolute value by $C(T)
\epsilon^{2}_n/n^3$. Since the third expression in Corollary \ref{l10}
is bounded by $C(T) \epsilon^{3}_n/n^3$, the proof is complete.
\end{proof}

\begin{lemma}
\label{l14}
Fix $x\in \bb T^d_n$, $1\le j\le d$ and $0\le t\le T$. We claim that
\begin{equation*}
\begin{aligned}
&  a_n\, (\, \nabla^n_j \, I^n_j \,)\, (t, x-e_j) \\
&\qquad \;=\;
\epsilon_n\, \mb m_j\, \sigma''_{j,j}(1/2)\, 
v (t,x/n) \, (\partial_{x_j} v)\, (t,x/n) \,
\;+\; \Big(\, \epsilon^2_n \,+\,
\frac {1}{n} \, \Big) \, R_n \;,
\end{aligned}
\end{equation*}
where $R_n$ is a remainder whose absolute value is bounded by
$C(T)$, where $C(T)$ is a finite constant which depends only on $T$
and on $v$ through the $L^\infty$ norm of its first three derivatives.
\end{lemma}

\begin{proof}
Let $d_j$ be the cylinder function defined by $d_j(\eta) = c_j(\eta)
\, \eta_0 \, [1-\eta_{e_j}]$. With this notation and since $c_j$ does
not depend on $\eta_0$, $\eta_{e_j}$, we may rewrite the
left-hand side of the statement of the lemma as
\begin{equation*}
\frac{n}{\epsilon_n}\, \mb m_j\, \Big\{ E_{\nu^n_{w(t)}} 
\big[\, \tau_x d_j \, \big] \,-\,  
E_{\nu^n_{w(t)}} \big[\, \tau_{x-e_j} \, d_j \,  \big] \Big\}\;.
\end{equation*}

Recall the definition of the measure $\nu^n_{w_{t,x}(\cdot)}$, introduced
just after \eqref{25}, and that $\nu^n_{w(t,x)}$ represents the
homogeneous product Bernoulli measure with density $w(t,x)$. By
Corollary \ref{l10} and since the absolute value of $c_{z,z'}$ is
bounded by $C(T)/n$, the previous expression is equal to
\begin{equation*}
\mb m_j\, \sum_{z} [\nabla^n_j v]\, (t,[x+z-e_j]/n)  \, 
E_{\nu^n_{w(t,x)}} \big[ \, d_j  \, \omega_z \, \big] \;+\; \frac
{\epsilon_n}{n} \,  R_n \;.
\end{equation*}
In this formula and below, $R_n$ is a remainder whose absolute value
is bounded by $C_0$, for some constant $C_0$ depending only on $T$ and
on the $L^\infty$ norm of the first three derivatives of $v$. The
exact expression of the remainder $R_n$ may change from line to line.

A Taylor expansion around $x/n$ yields that the previous sum is equal
to
\begin{equation*}
\mb m_j\, (\partial_{x_j} v)\, (t,x/n)  \, \sum_{z}
E_{\nu^n_{w(t,x)}} \big[ \, d_j  \, \omega_z \, \big] \;+\; \frac
{1}{n} \,  R_n \;.
\end{equation*}
By definition of $d_j$ and by \eqref{33}, $\widetilde{d}_j (\rho) =
\widetilde{c}_j (\rho) \, \rho \, [1-\rho] =
\sigma_{j,j}(\rho)$. Hence, by \eqref{26}, the sum over $z$ is equal
to $\sigma'_{j,j}(w (t,x/n))$. By \eqref{34} and a Taylor expansion,
this later expression is equal to $\epsilon_n\, \sigma''_{j,j}(1/2)\,
v (t,x/n) + \epsilon_n^2 R_n$. This completes the proof of the lemma. 
\end{proof}

\begin{proof}[Proof of Lemma \ref{l5}]
The proof is a straightforward consequence of Lemmata \ref{l13} and
\ref{l14} and from the fact that $v$ is the solution of the equation
\eqref{beq}. In both lemmata, the constant depends on
the $L^\infty$ norm of the first three derivatives of $v$. Lemma
\ref{l9} states that these derivatives are bounded by a constant which
depends on $v_0$.
\end{proof}

We conclude this section with some results used above.  Let $v: \bb
T^d \to \bb R$ be a function in $C^1(\bb T^d)$, and let $w: \bb T^d_n
\to \bb R$ be given by $w(x) = (1/2) + \epsilon_n v(x/n)$. Recall from
\eqref{23} that we denote by $\nu^n_{w(\cdot)}$ the product measure on
$\Omega_n$ in which the density of $\eta_x$ is $w(x/n)$, while
$\nu^n_{w(0)}$ represents the homogeneous product measure with
constant density equal to $w(0)$.

\begin{lemma}
\label{l19}
Let $g: \Omega_n \to \bb R$ be a local function. Then, there exists a
constant $C_0$, depending only on the cylinder function $g$ and on
$\Vert \nabla v\Vert_\infty$, such that
\begin{align*}
E_{\nu^n_{w(\cdot)}} \big[ g \big] \; & =\; E_{\nu^n_{w(0)}} \big[ g
\big] 
\;+\; \epsilon_n\, \sum_{z} [v(z/n) - v(0)] \, 
E_{\nu^n_{w(0)}} \big[ \, g  \, \omega_z \, \big] \\
\; & +\; \frac{1}{2}\, \epsilon^{2}_n\, 
\sum_{z\not = z'} [v(z/n) - v(0)] \, [v(z'/n) - v(0)] \, 
E_{\nu^n_{w(0)}} \big[ \, g  \, \omega_z\, \omega_{z'} \, \big] 
\; +\; R_n \;,
\end{align*}
where $|R_n| \le C_0 (\epsilon_n/n)^3$,
$\omega_z = [\eta_z - w(0)]/w(0) [1-w(0)]$. On the right hand side,
the sum is carried out over all $z$ (and $z'\not =z$) in the support
of $g$.
\end{lemma}

\begin{proof}
Fix a local function $g: \Omega_n \to \bb R$, and denote by
$\Lambda(g)$ its support. Clearly, as $\nu^n_{w(\cdot)}$,
$\nu^n_{w(0)}$ are product measures,
\begin{equation*}
E_{\nu^n_{w(\cdot)}} \big[ g \big] \; =\; 
E_{\nu^n_{w(0)}} \big[ \, g \, e^H \, \big] \;,
\end{equation*}
where 
\begin{align*}
H(\eta) \; & =\; \sum_{z\in \Lambda(g)} \eta_z \log \Big( 1 \,+\,
\frac{\epsilon_n \, [v(z/n) - v(0)]}{w(0)} \Big)  \\
\;& +\;
\sum_{z\in \Lambda(g)} [1-\eta_z] \log \Big( 1 \,-\,
\frac{\epsilon_n \, [v(z/n) - v(0)]}{1- w(0)} \Big) \;.
\end{align*}
The result follows from a Taylor expansion up to the third order.
\end{proof}

Recall from \eqref{24} the definition of the discrete partial
derivative in the $j$-th direction represented by $\nabla^n_j$,
and from \eqref{100} the definition of $c_{z,z'}$.

\begin{corollary}
\label{l10}
Let $g: \Omega_n \to \bb R$ be a local function. Then, there exists a
constant $C_0$, depending only on the cylinder function $g$ and on
$\Vert \nabla v\Vert_\infty$, such that
\begin{align*}
E_{\nu^n_{w(\cdot)}} \big[ \tau_{-e_j} g \,-\, g \big] \; & =\;
\frac {-\, \epsilon_n}{n} \, \sum_{z} [\nabla^n_j v]\, ([z-e_j]/n)  \, 
E_{\nu^n_{w(0)}} \big[ \, g  \, \omega_z \, \big] \\
\; & +\; \frac{\epsilon^{2}_n}{2\, n}\, \, 
\sum_{z\not = z'} c_{z,z'}  \, 
E_{\nu^n_{w(0)}} \big[ \, g  \, \omega_z\, \omega_{z'} \, \big] 
\; +\; R_n \;,
\end{align*}
where $|R_n| \le C_0 (\epsilon_n/n)^3$.
\end{corollary}

\begin{proof}
Fix  a local function $g: \Omega_n \to \bb R$. According to the
previous lemma, the expectation appearing on the left-hand side of the
statement is equal to
\begin{align*}
&\; \epsilon_n\, \sum_{z} [v(z/n) - v(0)] \, 
E_{\nu^n_{w(0)}} \big[ \, [\, \tau_{-e_j} g \,-\, g\,]   \, \omega_z \, \big] \\
& +\; \frac{1}{2}\, \epsilon^{2}_n\, 
\sum_{z\not = z'} [v(z/n) - v(0)] \, [v(z'/n) - v(0)] \, 
E_{\nu^n_{w(0)}} \big[ \, [\, \tau_{-e_j} g \,-\, g\,]  \, \omega_z\, \omega_{z'} \, \big] 
\; +\; R_n \;,
\end{align*}
where $|R_n| \le C_0 (\epsilon_n/n)^3$, for some constant $C_0$ which
depends only on $g$ and $\Vert \nabla v\Vert_\infty$. Here, the sum
over $z$ is carried out over all $z$ (and $z'\not =z$) in the support
of $\tau_{-e_j} g \,-\, g$. As the measure $\nu^n_{w(0)}$ is
homogeneous, a change of variables permits to complete the proof of
the lemma.  
\end{proof}

Let $g: \{0,1\}^{\bb Z^d} \to \bb R$ be a local function. Recall from
\eqref{28} the definition of the smooth function
$\widetilde g:[0,1] \to \bb R$. A similar computation to the one
presented in the proof of Lemma \ref{l19} yields that
\begin{equation}
\label{26}
\widetilde g' (\theta) \;=\; \sum_{z} 
E_{\nu_\theta} \big[ \, g  \, \omega_z \, \big] \;, \quad \theta\in[0,1]\;.
\end{equation}

Along the same lines, we may also prove the Einstein relation. 

\begin{proposition}
\label{l12}
For every $\alpha \in (0,1)$, $1\le j\le d$, 
\begin{equation*}
\widetilde{c}_j(\alpha) \;=\; \sum_{p=1}^{n_j} D_p(j, j) \, 
\widetilde{g}'_{j,p}(\alpha)  \quad\text{and}\quad
\sum_{p=1}^{n_j} D_p(j, k) \, 
\widetilde{g}'_{j,p}(\alpha) \;=\; 0 
\quad\text{for}\quad k \,\not =\, j \;.
\end{equation*}
\end{proposition}

\begin{proof}
Fix $1\le j\le d$, $\alpha\in (0,1)$ and let $u: \bb T^d \to (0,1)$ be
a differentiable function such that $u(0)=\alpha$,
$(\partial_{x_j} u)(0) \not = 0$.  Take the expectation with respect to
$\nu^n_{u(\cdot)}$ on both sides of \eqref{02}. 

First, note that $E_{\nu_\alpha} [ \, c_j(\eta) \, [\eta_0 - \eta_{e_j}] \, ] = 0$
since $c_j$ does not depend on $\eta_0$ and $\eta_{e_j}$.
For the left-hand side,
by the proof of Lemma \ref{l19} and since $u(0)=\alpha$,
\begin{equation*}
E_{\nu^n_{u(\cdot)}} \big[ \, c_j(\eta) \, [\eta_0 - \eta_{e_j}] \, \big] \;=\;
\sum_z [u(z/n) - \alpha] \, E_{\nu_\alpha} \big[ \, 
c_j(\eta) \, [\eta_0 - \eta_{e_j}] \, \omega_z \, \big] \;+\; O(1/n^2)\;,
\end{equation*}
where $\omega_z = [\eta_z - \alpha]/\alpha(1-\alpha)$.
Since $c_j$ does not depend on $\eta_0$ and $\eta_{e_j}$, for $z\neq 0,e_j$,
\begin{align*}
E_{\nu_\alpha} \big[ \, c_j(\eta) \, [\eta_0 - \eta_{e_j}] \, \omega_z \, \big]
\;=\;0\;.
\end{align*}
As $u(0)= \alpha$, the sum in the penultimate line is equal to 
\begin{equation*}
[u(e_j/n) - \alpha] \,
E_{\nu_\alpha} \big[ \, c_j(\eta) \, [\eta_0 - \eta_{e_j}] \, \omega_{e_j}
\, \big] \;=\; -\, [u(e_j/n) - \alpha] \,
E_{\nu_\alpha} \big[ \, c_j(\eta) \, \big] \;.
\end{equation*}

We turn to the expectation of the right-hand side of \eqref{02}. 
By the proof of Lemma \ref{l19} and since $\sum_y m_{j,p}(y)=0$, the
first term in the expansion vanishes so that
\begin{align*}
& \sum_{p=1}^{n_j} \sum_{y\in\bb Z^d} m_{j,p}(y) \, E_{\nu^n_{u(\cdot)}}
\big[ \, \tau_y \, g_{j,p} \, \big] \\
&\qquad \;=\;
\sum_{p=1}^{n_j} \sum_{y\in\bb Z^d} m_{j,p}(y) \, 
\sum_z [u(z/n) - \alpha] \, 
E_{\nu_\alpha} \big[ \,  (\tau_y \, g_{j,p})  \, \omega_z \, \big] 
\;+\; O(1/n^2)\;.
\end{align*}
A change of variables $\eta \mapsto \tau_y \eta$ and a Taylor expansion
permit to rewrite the sum as
\begin{equation*}
\frac 1n \sum_{p=1}^{n_j} \sum_{y\in\bb Z^d} m_{j,p}(y) \, 
\sum_z (z+y) \cdot (\nabla u)(0) \, 
E_{\nu_\alpha} \big[ \,  g_{j,p} \, \omega_z \, \big] \;+\; O(1/n^2) \;.
\end{equation*}
Since $\sum_y m_{j,p}(y)=0$ and, by definition, $\sum_y y_k \,
m_{j,p}(y)= - \, D_p(j,k)$, the last expression is equal to
\begin{equation*}
-\, \frac {1}n \sum_{p=1}^{n_j} \big[ \, D_p(j, \cdot) \, \cdot (\nabla u)(0) \, \big]
\widetilde{g}'_{j,p}(\alpha) \;+\; O(1/n^2) \;.
\end{equation*}
\\
Putting together the previous estimates, we conclude that for every
$v\in \bb R^d$,
\begin{equation*}
v_j \, \widetilde{c}_j(\alpha) \;=\;
\sum_{p=1}^{n_j} \sum_k D_p(j, k) \, v_k \, 
\widetilde{g}'_{j,p}(\alpha) \;.
\end{equation*}
This completes the proof of the proposition.
\end{proof}

\section{The adjoint generator}
\label{sec4}

Fix a function $\varrho: \bb T^d_n \to (0,1)$.  Throughout this
section, $\nu_\varrho$ is a product measure on $\Omega_n$ with
marginals given by $E_{\nu_\varrho}[\eta(x)] = \varrho(x)$, $x\in\bb
T^d_n$. Recall that we denote by $\chi(\alpha)$ the static
compressibility, $\chi(\alpha) \,=\, \alpha \, [\, 1 - \alpha \,]$.

For each $q\ge0$, recall the definition of the set $\ms E_q$: 
$\ms E_q=\{A\subset \T_n^d: |A|=k\}$.
Denote by $\mb P^{(q)}_\varrho (\tau_x f)$ the projection
of the cylinder function $\tau_x f$ over the linear set of functions
of degree $q$:
\begin{equation*}
[\, \mb P^{(q)}_\varrho (\tau_x f)\,]\, (\eta)  \;=\; \sum_{A\in \ms E_q} 
\mf f\, (x,A)\; \omega_\varrho (A+x)\;.
\end{equation*}
In particular, $\mb P^{(0)}_\varrho (\tau_x f) =  E_{\nu_\varrho}[\,
\tau_x f \,]$. Let $\mb P^{(+q)}_\varrho = \sum_{p\ge q} \mb
P^{(p)}_\varrho$ so that  
\begin{equation*}
[\, \mb P^{(+q)}_\varrho (\tau_x f)\,] \, (\eta) 
\;=\; \sum_{p\ge q} \sum_{A\in \ms E_p} 
\mf f\, (x,A)\; \omega_\varrho (A+x)\;.
\end{equation*}
We represent $\mb P^{(+1)}_\varrho$ by $\mb P_\varrho$: 
\begin{equation*}
[\, \mb P_\varrho \, (\tau_x f)\, ]\, (\eta) \;=\; 
(\tau_x  f)\, (\eta) \,-\, E_{\nu_\varrho}[\, \tau_x f \,]\;.
\end{equation*}

The statement of Lemma \ref{APl01} requires some notation. Recall from
\eqref{13} that $D_j$ stands for the difference operator, and from
\eqref{12} that we denote by $j_{x,x+e_j}$ the instantaneous current
over the bond $(x,x+e_j)$.

For $1\le j\le d$, $1\le p\le n_j$, $x\in \bb T^d_n$, let
\begin{align}
\label{09}
A_j(x) \;=\; \frac{\chi (\varrho(x)) \,+\, \chi(\varrho(x+e_j))}
{\chi (\varrho(x)) \; \chi(\varrho(x+e_j))}\;, 
\end{align}
\begin{equation*}
B^{(1)}_{j,p} (x) \;=\;  \frac 12\, \sum_{y\in\T_n^d} m_{j,p}(y) \,
A_j(x-y)\, (D_j \varrho) (x-y) \; ,  
\end{equation*}
\begin{equation*}
E^{(1)}_j (x) \;=\;  \frac 12\, A_j(x)\, [\, (D_j \varrho) (x) \,]^2
\;, \quad F^{(1)}_j (x) \;=\; 
\frac {[\, D_j (\chi\circ \varrho)] (x) \, (D_j \varrho) (x)}
{2\, \chi(\varrho(x+e_j))}  \;, 
\end{equation*}
\begin{equation*}
G^{(1)}_j (x) \;=\; 
\frac {[\,D_j (\chi\circ \varrho)] (x) \, (D_j \varrho) (x)}
{2\, \chi(\varrho(x))}  \;\cdot 
\end{equation*}
Finally, for $A\subset \bb T_n^d$, let
\begin{equation}
\label{41}
H^{(1)}_j(\varrho, x, A) \; =\; E^{(1)}_j(x)\; \mf c_{j}(x,A)  
\;+\; \sum_{p=1}^{n_j} B^{(1)}_{j,p} (x) \; \mf g_{j,p}(x,A) \;+\;
J^{(1)}_j(x, A)\;,
\end{equation}
where
\begin{align*}
J^{(1)}_j(x, A) \; =\; & -\; \Upsilon_{\{0,e_j\}} (A)\; 
[\, (D_j \varrho)\, (x)\,]^2 \; \mf c_j(x,A\setminus \{0,e_j\}) \\
& +\; \Upsilon_{\{0\}} (A) \;
F^{(1)}_j \, (x) \; \mf c_j(x,A\setminus \{0\}) \\
& +\; \Upsilon_{\{e_j\}} (A) \;
G^{(1)}_j \, (x) \; \mf c_j(x,A\setminus \{e_j\}) \;.
\end{align*}
In this formula, $\mf c_{j}(x,A)$, $\mf g_{j,p} (x,A)$ represent the
Fourier coefficients, introduced in \eqref{04}, of the cylinder
functions $c_j$, $g_{j,p}$, respectively; and $\Upsilon_B$ stand for
the function introduced in \eqref{14}.

It follows from \eqref{05} that there exists $\ell\ge 1$ such that
$H^{(1)}_j(\varrho, x, A)=0$ if $A\not\subset \Lambda_\ell$. Note that
the functions of $x$ which appear in the previous formula either
contain the product of derivatives [this is the case of $E^{(1)}_j$,
$F^{(1)}_j$ and $G^{(1)}_j$] or a mean-zero sum of  discrete derivatives, which is
the case of $B^{(1)}_{j,p}$. This structure makes $n^2H_j^{(1)}(\varrho, x, A)$
bounded in $n$ if the reference density is good enough since these derivatives
absorb the speeded-up factor $n^2$.

\begin{lemma}
\label{APl01}
Denote by $L_{n, \nu_\varrho}^{S,*}$ the adjoint of $L_n^S$ in $L^2(\nu_\varrho)$.
Then,
\begin{align*}
L_{n, \nu_\varrho}^{S,*} \, {\bf 1} \;& =\; \sum_{j=1}^d
\sum_{x\in\T_n^d} \big\{ \, E_{\nu_\varrho} \big[\, j_{x-e_j,x}\big] \,-\,  
E_{\nu_\varrho} \big[\, j_{x,x+e_j}\big]  \, \big\}
\; \omega_\varrho (x) \\
& + \; \sum_{j=1}^d \sum_{A: |A|\ge 2}  \sum_{x\in\T_n^d} H^{(1)}_j(\varrho,
  x, A) \, \omega_\varrho (A+x)\;, 
\end{align*}
where the (finite) sum over $A$ is performed over finite subsets $A$
with at least two elements.
\end{lemma}

Note that the first term on the right-hand side contains only terms of
degree $1$, while the second one only terms of degree $2$ or higher. 

The proof of this lemma is divided in four Lemmata and one
identity, presented in \eqref{08}. We first compute the adjoint
$L_{n,\nu_\varrho}^{S,*}$ of $L^S_n$.

\begin{lemma}
\label{A4}
For $x\in\T_n^d$ and $1\le j\le d$, let
\begin{equation*}
J_{x,x+e_j}(\eta) \;=\; \frac{\nu_\varrho (\sigma^{x,x+e_j}\eta)}
{\nu_\varrho (\eta)}\;\cdot
\end{equation*} 
Then, for any $f\in L^2(\nu_\varrho)$,
\begin{align*}
(\, L_{n,\nu_\varrho}^{S,*} \, f\,)\, (\eta) \; &=\;
\sum_{x\in\T_n^d} \sum_{j=1}^d c_j (\tau_x \eta)\, J_{x,x+e_j}(\eta) \, 
\{\, f(\sigma^{x,x+e_j}\eta) \,-\, f(\eta)\, \} \\
&+\; \sum_{x\in\T_n^d} \sum_{j=1}^d c_j (\tau_x \eta)\, \{\, J_{x,x+e_j}(\eta) 
\,-\, 1\,\} \,  f(\eta)  \;.
\end{align*}
\end{lemma}

The proof of this lemma is elementary and left to the reader.

\begin{lemma}
\label{A5}
We have that
\begin{align*}
(L_{n,\nu_\varrho}^{S,*} \, {\bf 1})(\eta) \;
& =\; \sum_{j=1}^d \sum_{x\in\T_n^d} 
\Big\{ \, E_{\nu_\varrho} \big[\, j_{x-e_j,x} \, \big] \,-\,  
E_{\nu_\varrho} \big[\, j_{x,x+e_j} \,  \big]  \, \Big\}
\, \omega_\varrho (x)  \\
& +\; \sum_{x\in\T_n^d} \sum_{j=1}^d (\mb P_\varrho \tau_x c_j) (\eta)\;
(D_j \varrho)\, (x)
\; [\, \omega_\varrho (x) \; -\; \omega_\varrho (x+e_j)\,] \\
& -\; \sum_{x\in\T_n^d} \sum_{j=1}^d c_j (\tau_x \eta)\, 
[\, (D_j \varrho) \, (x) \,]^2 \, \omega_\varrho (x) \, 
\omega_\varrho (x+e_j) \;.
\end{align*}
\end{lemma}

\begin{proof}
By Lemma \ref{A4}, 
\begin{equation*}
L_{n,\nu_\varrho}^{S,*} \, {\bf 1}\;=\; 
\sum_{x\in\T_n^d} \sum_{j=1}^d c_j (\tau_x \eta)\, \{\, J_{x,x+e_j}(\eta) 
\,-\, 1\,\}  \;.
\end{equation*}
The definition of $J_{x,x+e_j}$ and a straightforward computation
yield that this expression is equal to 
\begin{equation*}
\sum_{x\in\T_n^d} \sum_{j=1}^d c_j (\tau_x \eta)\, 
(D_j \varrho)\, (x)\,
\Big \{\, \frac{\eta_x \, (1-\eta_{x+e_j})}
{\varrho (x) \, [1-\varrho (x+e_j)]} \;-\;
\frac{\eta_{x+e_j} \, (1-\eta_x)}
{\varrho (x+e_j) \, [1-\varrho (x)]} \,\Big\}  \;.
\end{equation*}
Recall that
$\omega_\varrho (x) = [\eta(x) - \varrho(x)]/\chi(\varrho(x))$. The
expression inside braces can be written as
\begin{align*}
\omega_\varrho (x) \; -\; \omega_\varrho (x+e_j) \; -\;
(D_j \varrho) \, (x) \, \omega_\varrho (x) \, 
\omega_\varrho (x+e_j) \;.
\end{align*}
Therefore, 
\begin{align*}
(L_{n,\nu_\varrho}^{S,*} \, {\bf 1})(\eta) \;
& =\; \sum_{x\in\T_n^d} \sum_{j=1}^d  E_{\nu_\varrho} [\,c_j (\tau_x \eta)\,]\;
(D_j \varrho)\, (x)
\; [\, \omega_\varrho (x) \; -\; \omega_\varrho (x+e_j)\,] \\
& +\; \sum_{x\in\T_n^d} \sum_{j=1}^d (\mb P_\varrho \tau_x c_j) (\eta)\;
(D_j \varrho)\, (x)
\; [\, \omega_\varrho (x) \; -\; \omega_\varrho (x+e_j)\,] \\
& -\; \sum_{x\in\T_n^d} \sum_{j=1}^d c_j (\tau_x \eta)\, 
[\, (D_j \varrho)\, (x)\,]^2 \, \omega_\varrho (x) \, 
\omega_\varrho (x+e_j) \;.
\end{align*}
Note that the second and third lines contain only terms of degree $2$
or more, while the first line have only terms of degree $1$.

Since $c_j$ does not depend on $\eta(0)$ and $\eta(e_j)$, by
definition of the instantaneous current $j_{x,x+e_j}$,
\begin{equation*}
E_{\nu_\varrho} [\,c_j (\tau_x \eta)\,]\;
(D_j \varrho)\, (x)  \;= \; -\, 
E_{\nu_\varrho} \big[\,c_j (\tau_x \eta)\, 
[ \, \eta(x) \,-\, \eta(x+e_j)\,]  \,  \big]\;= \; -\, 
E_{\nu_\varrho} \big[\, j_{x,x+e_j} \,  \big]\;.
\end{equation*}
To complete the proof, it remains to insert this expression in the
first line of the formula for
$(L_{n,\nu_\varrho}^{S,*} \, {\bf 1})(\eta)$ and to sum by parts.
\end{proof}

In view of \eqref{03}, the third term of Lemma \ref{A5} can
be written as
\begin{equation*}
-\; \sum_{j=1}^d \sum_{x\in\T_n^d}  \sum_A 
[\, (D_j \varrho)\, (x)\,]^2 \, \mf c_j(x,A) \, \omega_\varrho (A+x)
\, \omega_\varrho (x) \, 
\omega_\varrho (x+e_j) \;,
\end{equation*}
where $\mf c_j(x,A)$ stands for the Fourier coefficients of
$\tau_x c_j$, given by \eqref{04}. As $c_j$ does not depend on
$\eta(0)$ and $\eta(e_j)$, $\mf c_j(x,A)=0$ if $A$ contains $0$ or
$e_j$. We may therefore restrict the sum to sets which do not contain
these points and rewrite the previous expression as
\begin{equation}
\label{08}
\begin{aligned}
& -\; \sum_{j=1}^d \sum_{x\in\T_n^d}  \sum_{A: A \cap\{0,e_j\}= \varnothing}
[\, (D_j \varrho)\, (x)\,]^2 \, \mf c_j(x,A) \, 
\omega_\varrho (\, [\, A \cup \{0,e_j\}\, ]+x\, ) \\
&\quad =\; 
-\; \sum_{j=1}^d \sum_{x\in\T_n^d}  \sum_{A: A \supset \{0,e_j\}}
[\, (D_j \varrho)\, (x)\,]^2 \, \mf c_j(x,A\setminus \{0,e_j\}) \, 
\omega_\varrho (\, A + x\, )\;.
\end{aligned}
\end{equation}

We turn to the second term of Lemma \ref{A5}. 

\begin{lemma}
\label{A1}
For each $1\le j\le d$,
\begin{align}
\label{06}
& \sum_{x\in\T_n^d} [\, \mb P_\varrho \, (\tau_x \, c_j)\,]\,  (\eta)\;
(D_j \varrho) (x)
\; [\, \omega_\varrho (x) \; -\; \omega_\varrho (x+e_j)\,] \\
& \quad =\; 
\frac 12\, \sum_{x\in\T_n^d} 
[\, \mb P^{(+2)}_\varrho \,(\tau_x \, j_{0,e_j}) \, ] \, (\eta)\;
A_j(x)\, (D_j \varrho) (x)  \nonumber\\
& \quad +\, \frac 12\, \sum_{x\in\T_n^d} 
[\, \mb P^{(+2)}_\varrho (\tau_x \, c_j ) \, ]\,  (\eta)\; 
A_j(x)\, [\, (D_j \varrho) (x) \,]^2  \nonumber \\
& \quad +\; \frac 12\, \sum_{x\in\T_n^d} 
[\, \mb P_\varrho (\tau_x \, c_j)\,]\,  (\eta)\; 
\Big\{ \frac {\omega(x)}{\chi(\varrho(x+e_j))}  
\,+\, \frac{\omega (x+e_j)}{\chi (\varrho(x))} \,\Big\}\, 
[D_j (\chi\circ \varrho)] (x) \, (D_j \varrho) (x)\;.
\nonumber
\end{align}
\end{lemma}

\begin{proof}
Recall the definition of $\xi_\varrho(x): \xi_\varrho(x)=\eta(x)-\varrho(x), x\in\T_n^d$.
Fix $j$ and write $\omega_\varrho (x) \; -\; \omega_\varrho
(x+e_j)$ as
\begin{align}\label{50}
& \frac 1{2\,\chi (\varrho(x)) \, \chi(\varrho(x+e_j))}\, 
[\, \xi_\varrho(x) \,-\, \xi_\varrho(x+e_j)\,]\, 
[\, \chi(\varrho(x+e_j)) \,+\, \chi(\varrho(x)) \,] \\
&\quad +\; \frac 1{2\,\chi (\varrho(x)) \, \chi(\varrho(x+e_j))}\, 
[\, \xi_\varrho(x) \,+\, \xi_\varrho (x+e_j)\,]\, 
[\, \chi(\varrho(x+e_j)) \,-\, \chi(\varrho(x)) \,]\;.\notag
\end{align}
On the other hand, taking the operator $\mb P_\varrho\circ \tau_x$
for \eqref{12}, one can obtain
\begin{align}\label{51}
[\, \mb P_\varrho \,(\tau_x \, j_{0,e_j})\,]\, (\eta) \;&=\; 
E_{\nu_\varrho}[\, \tau_x c_j\, ] \; [\, \xi_\varrho(x) \,-\, \xi_\varrho(x+e_j)\,]\\
& \quad +\,  [\,\mb P_\varrho\, (\tau_x\,  c_j )\,]\,  (\eta) \, [\, \xi_\varrho(x) \,-\, \xi_\varrho(x+e_j)\,] \notag \\
& \qquad-\,  [\,\mb P_\varrho\, (\tau_x\,  c_j )\,]\,  (\eta) \, (D_j\varrho)(x)\;. \notag
\end{align}

From \eqref{50} and \eqref{51}, the left-hand side of \eqref{06} becomes
\begin{align*}
& \frac 12\, \sum_{x\in\T_n^d} [\,
\mb P_\varrho \,(\tau_x \, j_{0,e_j})\,]\, (\eta)\;
A_j(x)\, (D_j \varrho) (x)  \\
& \quad -\, \frac 12\, \sum_{x\in\T_n^d} 
E_{\nu_\varrho}[\, \tau_x c_j\, ] \; [\, \xi_\varrho(x) \,-\, \xi_\varrho(x+e_j)\,]\,
A_j(x)\, (D_j\,  \varrho) (x)  \\
& \quad +\, \frac 12\, \sum_{x\in\T_n^d} 
[\,\mb P_\varrho\, (\tau_x\,  c_j )\,]\,  (\eta)\; 
A_j(x)\, [\, (D_j \varrho) (x) \,]^2  \;+\; L_3 \;,
\end{align*}
where $L_3$ is the last term appearing on the right-hand side of
\eqref{06} and $A_j(x)$ has been introduced in \eqref{09}.

Since $c_j$ does not depend on $\eta(0)$, $\eta(e_j)$, the expectation
with respect to $\nu_\varrho$ of the left-had side of \eqref{06}
vanishes. It is also clear that the covariance of this sum with
respect to $\xi_\varrho(z)$ vanishes for all $z\in \bb T^d_n$. We may
therefore introduce the operator $\mb P^{(+2)}_\varrho$ in front of the
sum. By doing so, the second sum of the previous formula vanishes
because it contains only terms of degree $1$. This completes the proof
of the lemma.
\end{proof}

We further express the sums on the right-hand side of \eqref{06} in
terms of the Fourier coefficients of the cylinder functions.  Recall
the notation introduced in \eqref{09} and below.

\begin{lemma}
\label{A2}
For each $1\le j\le d$,
\begin{align}
\label{10}
& \sum_{x\in\T_n^d} [\, \mb P_\varrho \, (\tau_x \, c_j)\,]\,  (\eta)\;
(D_j \varrho) (x)
\; [\, \omega_\varrho (x) \; -\; \omega_\varrho (x+e_j)\,] \\
& \quad =\; 
\sum_{A: |A|\ge 2} \sum_{x\in\T_n^d}  \sum_{p=1}^{n_j} 
B^{(1)}_{j,p} (x) \, \mf g_{j,p}(x,A) \, 
\omega_\varrho (\, A + x\, )  \nonumber \\
& \quad +\,  \sum_{A: |A|\ge 2} \sum_{x\in\T_n^d} 
E^{(1)}_j(x)\, \mf c_{j}(x,A) \, \omega_\varrho (\, A + x\, ) \; 
 \nonumber \\
& \quad +\; \sum_{\substack{A: |A|\ge 2 \\ A \ni 0}} \sum_{x\in\T_n^d}
  F^{(1)}_j (x)\,
\mf c_{j}(x,A\setminus \{0\}) \, \omega_\varrho (\, A + x\, ) 
\nonumber \\
& \quad +\; \sum_{\substack{A: |A|\ge 2 \\ A \ni e_j}}
  \sum_{x\in\T_n^d} G^{(1)}_j (x)\,
\mf c_{j}(x,A\setminus \{e_j\}) \, \omega_\varrho (\, A + x\, ) \;.
\nonumber
\end{align}
\end{lemma}

\begin{proof}
Fix $1\le j\le d$.  We consider separately each term on the right-hand
side of \eqref{06}. Let $B_j(x) = A_j(x)\, (D_j \varrho) (x)$. By the
gradient conditions \eqref{02}, the first term can be written as
\begin{equation*}
\frac 12\, \sum_{x\in\T_n^d} \sum_{p=1}^{n_j} \sum_{y\in\T_n^d} m_{j,p}(y) 
\, [\, \mb P^{(+2)}_\varrho \,(\tau_{x+y} \, g_{j,p}) \, ] \, (\eta)\;
B_j(x)\;.
\end{equation*}
Perform the change of variables $x'=x+y$ and express the cylinder
function $g_{j,p}$ in terms of its Fourier coefficients to rewrite
this expression as
\begin{equation*}
\frac 12\, \sum_{x\in\T_n^d} \sum_{p=1}^{n_j} 
\Big( \sum_{y\in\T_n^d} m_{j,p}(y) B_j(x-y) \Big)
\sum_{A: |A|\ge 2} \mf g_{j,p}(x,A) \, 
\omega_\varrho (\, A + x\, ) \;.
\end{equation*}
This expression corresponds to the first one on the right-hand side of
\eqref{10}. The other three can be obtained easily.
\end{proof}

Recall the definition of the asymmetric part of the generator
introduced in \eqref{11}.  For $1\le j\le d$, let
$C_j$, $I_j: \bb T^d_n \to \bb R$ be given by
\begin{equation}
\label{15}
C_j(x) \;=\; \mb m_j\, \varrho(x) \, [1 - \varrho(x+e_j)]\;,
\quad I_j(x) \;=\; E_{\nu_\varrho}[\tau_x c_j] \, C_j(x)
\;.
\end{equation}
For $1\le j\le d$, $1\le p\le n_j$, $x\in \bb T^d_n$, let
\begin{align}\label{16}
B^{(2)}_{j,p} (x) \;&=\; -\, \frac 12\, \sum_{y\in\T_n^d} m_{j,p}(y) \,
A_j(x-y)\, (C_j \varrho) (x-y) \; ,  \notag\\
E^{(2)}_j (x) \;&=\; -\, \frac 12\, A_j(x)\, (D_j \varrho) (x) \,C_j(x)\;, \notag\\
F^{(2)}_j (x) \;&=\; -\, \frac {[\, D_j (\chi\circ \varrho)] (x) \, (C_j \varrho) (x)}
{2\, \chi(\varrho(x+e_j))}  \;,  \notag\\
G^{(2)}_j (x) \;&=\; -\, \frac {[\,D_j (\chi\circ \varrho)] (x) \, (C_j \varrho) (x)}
{2\, \chi(\varrho(x))}  \;\cdot 
\end{align}

For $A\subset \bb T_n^d$, let
\begin{equation}
\label{42}
H^{(2)}_j(\varrho, x, A) \; =\; E^{(2)}_j(x)\; \mf c_{j}(x,A)  
\;+\; \sum_{p=1}^{n_j} B^{(2)}_{j,p} (x) \; \mf g_{j,p}(x,A) \;+\;
J^{(2)}_j(x,A)\;, 
\end{equation}
where
\begin{align*}
J^{(2)}_j(x,A) \; & =\; \Upsilon_{\{0,e_j\}} (A)\; 
(D_j \varrho)\, (x)\, C_j(x) \; \mf c_j(x,A\setminus \{0,e_j\}) \\
&  +\; \Upsilon_{\{0\}} (A) \;
F^{(2)}_j \, (x) \; \mf c_j(x,A\setminus \{0\}) \\
& +\; \Upsilon_{\{e_j\}} (A) \;
G^{(2)}_j \, (x) \; \mf c_j(x,A\setminus \{e_j\}) \;.
\end{align*}

\begin{lemma}
\label{APl02}
Let $L_{n,\nu_\varrho}^{T,*}$ be the adjoint of $L_n^T$ in
$L^2(\nu_\varrho)$. Then,
\begin{align*}
L_{n, \nu_\varrho}^{T,*} \, {\bf 1} \;=\; 
& -\, \sum_{x\in\T_n^d}\sum_{j=1}^d  
(\, D_j \, I_j \,)\, (x-e_j) \;
\omega_\varrho(x) \\
& + \; \sum_{j=1}^d \sum_{A: |A|\ge 2}  \sum_{x\in\T_n^d} H^{(2)}_j(\varrho,
  x, A) \, \omega_\varrho (A+x)\;, 
\end{align*}
where the (finite) sum over $A$ is performed over finite subsets $A$
with at least two elements.
\end{lemma}

The proof of this lemma relies on the next two lemmata.

\begin{lemma}
\label{A7}
Recall the definition of $J_{x,x+e_j}$ given in Lemma \ref{A4}. Then,
for any $f\in L^2(\nu_\varrho)$,
\begin{align*}
(\, L_{n,\nu_\varrho}^{T,*} \, f\,)\, (\eta) \; &=\;
\sum_{x\in\T_n^d} \sum_{j=1}^d \mb m_j\, (\tau_x c_j)(\eta)\, J_{x,x+e_j}(\eta)\, 
(1-\eta_{x})\, \eta_{x+e_j}\, \{f(\sigma^{x,x+e_j}\eta) - f(\eta) \} \\
&+\; \sum_{x\in\T_n^d} \sum_{j=1}^d \mb m_j\, (\tau_x c_j)(\eta)\,
\big\{(1-\eta_x)\, \eta_{x+e_j}\, J_{x,x+e_j}(\eta) 
\,-\, \eta_x\, (1-\eta_{x+e_j}) \big\} \,  f(\eta)  \;.
\end{align*}
\end{lemma}

\begin{lemma}
\label{A8}
We have that
\begin{align*}
L_{n,\nu_\varrho}^{T,*}{\bf 1} \;
& =\;  -\, \sum_{x\in\T_n^d}\sum_{j=1}^d  
(\, D_j \, I_j \,)\, (x-e_j) \,
\omega_\varrho(x) \\
& - \;  \sum_{x\in\T_n^d}\sum_{j=1}^d \,  [\mb P_\varrho (\tau_x
  c_j)]\, (\eta)\, C_j(x) \,
[\omega_\varrho(x) - \omega_\varrho(x+e_j)] \\
& +\; \sum_{x\in\T_n^d} \sum_{j=1}^d  (\tau_x c_j)(\eta)\,
C_j(x) \, (D_j \varrho) (x)\,
\omega_\varrho(x) \, \omega_\varrho(x+e_j)  \;.
\end{align*}
where $I_j(x) = E_{\nu_\varrho}[\tau_x c_j] \, C_j(x)$.
\end{lemma}

\begin{proof}
Recall the definition of $C_j$.  It follows from the previous
lemma and a straightforward computation that
\begin{align*}
L_{n,\nu_\varrho}^{T,*}{\bf 1} \;& =\; 
\sum_{x\in\T_n^d}\sum_{j=1}^d  (\tau_x c_j)(\eta)\,
C_j(x) \, [\omega_\varrho(x+e_j) - \omega_\varrho(x)] \\
& +\; \sum_{x\in\T_n^d} \sum_{j=1}^d  (\tau_x c_j)(\eta)\,
C_j(x) \,  (D_j \varrho) (x)\,
\omega_\varrho(x) \, \omega_\varrho(x+e_j)  \;.
\end{align*}
It remains to add and subtract $E_{\nu_\varrho}[\tau_x c_j]$ in the
first term and to sum by parts.
\end{proof}

\begin{proof}[Proof of Lemma \ref{APl02}]
The expression of $L_{n,\nu_\varrho}^{T,*}{\bf 1}$ is similar
to the one of $L_{n,\nu_\varrho}^{S,*}{\bf 1}$. In the second and third
terms one has to replace $D_j \varrho$ by $-\, C_j$. We may thus follow
the arguments presented for the symmetric part to complete the proof
of Lemma \ref{APl02}.
\end{proof}

\smallskip\noindent{\bf Acknowledgments.}  Part of this work was done
during K. Tsunoda's visit to IMPA.  He would like to thank IMPA for
numerous support and warm hospitality during his visit.   M. Jara acknowledges CNPq for its
support through the Grant 305075/2017-9, FAPERJ for its support
through the Grant E-29/203.012/2018 and ERC for its support through
the European Unions Horizon 2020 research and innovative programme
(Grant Agreement No. 715734). C. Landim has been partially supported
by FAPERJ CNE E-26/201.207/2014, by CNPq Bolsa de Produtividade em
Pesquisa PQ 303538/2014-7, and by ANR-15-CE40-0020-01 LSD of the
French National Research Agency. K. Tsunoda has been partially supported
by JSPS KAKENHI, Grant-in-Aid for Early-Career Scientists 18K13426.
The authors are grateful to the anonymous referees for the careful reading and their comments.

\smallskip\noindent{\bf Conflict of Interest:} The authors declare
that they have no conflict of interest.


\begin{thebibliography}{99}

\bibitem{bl} J. Beltr\'an, C. Landim: A lattice gas model for the
  incompressible Navier-Stokes equation.  Ann. Inst. Henri Poincar\'e
  Probab. Stat. {\bf 44}, 886--914 (2008).

%\bibitem{cl} E. Chavez, C. Landim: A correction to the hydrodynamic
%  limit of boundary driven weakly asymmetric exclusion processes in a
%  quasi-static time scale. J. Stat. Phys. {\bf 163}, 1079--1107
%  (2006).

\bibitem{d} R. L. Dobrushin: Caricature of Hydrodynamics. Proceed.
  IX--th International Congress of Math--Phys., , pag.117--132, 17--27
  July 1988, Simon, Truman, Davies ed., Adam Hilger 1989.

\bibitem{dpst88} R. L. Dobrushin, A.  Pellegrinotti, Yu.  M.  Suhov, L.
  Triolo: One-Dimensional harmonic lattice caricature of
  hydrodynamics: second approximation. J. Stat. Phys.
  {\bf 52}, 423--439 (1988).

\bibitem{dps90} R. L. Dobrushin, A.  Pellegrinotti, Yu.  M.  Suhov:
  One-dimensional harmonic lattice caricature of hydrodynamics: A
  higher correction.J. Stat. Phys. {\bf 61}, 387--402 (1990).

\bibitem{emy1} R. Esposito, R. Marra, H.-T. Yau: Diffusive limit of
  asymmetric simple exclusion.  Special issue dedicated to Elliott
  H. Lieb.  Rev. Math. Phys. {\bf 6}, 1233--1267 (1994).
 
\bibitem{emy2} R. Esposito, R. Marra, H.-T. Yau: Navier-Stokes
  equations for stochastic lattice gases. Commun.  Math. Phys. {\bf
    182}, 395--456 (1996).

 \bibitem{ft18} T. Funaki, K. Tsunoda: Motion by mean curvature from
  Glauber-Kawasaki dynamics. J. Stat. Phys. {\bf 177}, 183--208 (2019).

\bibitem{jl} M. Jara, C. Landim: The stochastic heat equation as limit
  of exclusion plus voter model, preprint.

\bibitem{jm1} M. Jara, O. Menezes: Non-equilibrium fluctuations for a
reaction-diffusion model via relative entropy, arXiv:1810.03418 (2018).

\bibitem{jm2} M. Jara, O. Menezes: Nonequilibrium fluctuations of
  interacting particle systems, arXiv:1810.09526 (2018).

\bibitem{kl} C. Kipnis, C. Landim: {\it Scaling Limits of Interacting
    Particle Systems}, Grundlheren der mathematischen Wissenschaften
  {\bf 320}, Springer-Verlag, Berlin, New York, 1999.

\bibitem{klo95} C. Kipnis, C. Landim, S. Olla: Macroscopic properties
  of a stationary non--equilibrium distribution for a non--gradient
  interacting particle system. Annales de l'Institut Henri Poincar\'e,
  s\'erie B, {\bf 31}, 191--221, (1995).
  
  \bibitem{klo} T. Komorowski, C. Landim, S. Olla: {\it Fluctuations in Markov processes. 
Time symmetry and martingale approximation.}, Grundlheren der mathematischen Wissenschaften
  {\bf 345}, Springer, Heidelberg, 2012.

\bibitem{loy96} C. Landim , S. Olla, H. T. Yau: Some properties of
  the diffusion coefficient for asymmetric simple exclusion processes.
  Ann. Probab. {\bf 24}, 1779--1807, (1996).

\bibitem{loy97} C. Landim, S. Olla, H. T. Yau: First-order correction
  for the hydrodynamic limit of asymmetric simple exclusion processes
  in dimension $d \ge 3$. Commun. Pure and Appl. Math. {\bf L},
  149--203 (1997).

\bibitem{lsv04} C. Landim, R. M. Sued, G. Valle: Hydrodynamic limit of
  asymmetric exclusion processes under diffusive scaling in $d\ge 3$.
  Commun. Math. Phys. {\bf 249}, 215--247, (2004).  

\bibitem{lsu1968} O. A Ladyzhenskaya, V. A. Solonnikov,
  N. N. Uraltseva: {\it Linear and quasilinear equations of parabolic
    type}, Translations of mathematical monographs, v. 23, American
  Mathematical Society, Providence, R.I., 1968.

%\bibitem{m} O. Menezes: Non-equilibrium fluctuations of interacting
%  particle systems, Thesis, IMPA, 2017.

\bibitem{q92} J. Quastel: Diffusion of color in the simple exclusion
  process. Comm. Pure Appl. Math. {\bf 45}, 623--679 (1992).

\bibitem{qy98} J. Quastel, H.-T. Yau: Lattice gases, large deviations,
and the incompressible Navier-Stokes equations.
Ann. Math. {\bf 148}, 51--108 (1998).

\bibitem{rez90} F. Rezakhanlou: Hydrodynamic limit for attractive
particle systems on ${\bf Z }^d$.
Comm. Math. Phys. {\bf 140}, 417--448 (1991).

\bibitem{v93} S. R. S. Varadhan: Nonlinear diffusion limit for a
  system with nearest neighbor interactions II, pp. 75--128 in:
  Asymptotic Problems in Probability Theory: Stochastic Models and
  Diffusions on Fractals (Sanda/Kyoto, 1990), K. D. Elworthy and
  N. Ikeda, eds., Pitman Res. Notes Math. Ser. 283, Longman
  Sci. Tech., Harlow, 1993.

%\bibitem{y} H.-T. Yau: Logarithmic Sobolev inequality for generalized
%  simple exclusion processes.  Probab. Theory Related Fields {\bf
%    109}, 507--538 (1997).
 
\end{thebibliography}
\end{document}